\documentclass{amsart}

\usepackage{graphicx}        
\usepackage{amsmath}
\usepackage{amssymb}
\usepackage{amsfonts}
\usepackage{amsthm}
\usepackage{amscd}
\numberwithin{equation}{section}
\theoremstyle{plain}
\newtheorem{theorem}{Theorem}[section]
\newtheorem{proposition}[theorem]{Proposition}
\newtheorem{corollary}[theorem]{Corollary}
\newtheorem{conjecture}[theorem]{Conjecture}
\newtheorem{lemma}[theorem]{Lemma}
\theoremstyle{definition}
\newtheorem{definition}[theorem]{Definition}
\newtheorem{remark}[theorem]{Remark}
\newtheorem*{example}{Example}
\copyrightinfo{2015}{Martin Schlichenmaier}

\newcommand{\refE}[1]{(\ref{eq:#1})}

\newcommand{\refS}[1]{Section~\ref{sec:#1}}
\newcommand{\refA}[1]{Appendix~\ref{sec:#1}}
\newcommand{\refT}[1]{Theorem~\ref{T:#1}}

\newcommand{\refP}[1]{Proposition~\ref{P:#1}}
\newcommand{\refD}[1]{Definition~\ref{D:#1}}

\newcommand{\refL}[1]{Lemma~\ref{L:#1}}

\newcommand{\C}{\ensuremath{\mathbb{C}}}
\newcommand{\N}{\ensuremath{\mathbb{N}}}

\newcommand{\Z}{\ensuremath{\mathbb{Z}}}
\newcommand{\K}{\ensuremath{\mathcal{K}}}

\newcommand{\J}{\ensuremath{\mathbb{J}}}
\newcommand{\Jl}{\ensuremath{\mathbb{J}_\lambda}}
\newcommand{\Pro}{\ensuremath{\mathbb{P}}}
\newcommand{\poc}{\ensuremath{{\mathbb{P}}^1(\C)}}

\newcommand{\cins}{\frac 1{2\pi\mathrm{i}}\int_{C_S}}

\newcommand{\cint}[1]{\frac 1{2\pi\mathrm{i}}\int_{#1}}

\newcommand{\g}{\mathfrak{g}}

\newcommand{\gb}{\overline{\mathfrak{g}}}

\newcommand{\gh}{\widehat{\mathfrak{g}}}

\newcommand{\A}{\mathcal{A}}
\newcommand{\W}{\mathcal{W}}
\newcommand{\V}{\mathcal{V}}

\newcommand{\La}{\mathcal{L}}
\renewcommand{\L}{\mathcal{L}}

\newcommand{\ord}{\operatorname{ord}}

\newcommand{\res}{\operatorname{res}}

\newcommand{\fpz}{\frac {d }{dz}}
\newcommand{\pfz}[1]{\frac {d#1}{dz}}

\newcommand{\PGL}{\mathrm{PGL}}
\newcommand{\Ho}{\mathrm{H}}

\newcommand{\al}{\ensuremath{\alpha}}
\newcommand{\be}{\ensuremath{\beta}}

\newcommand{\Fl}[1][\lambda]{\mathcal{F}^{#1}}

\newcommand{\Fln}[1][n]{\mathcal{F}_{#1}^\lambda}

\newcommand{\Ah}{\widehat{\mathcal{A}}}

\newcommand{\Do}{\mathcal{D}^1}

\newcommand{\Dhg}{\widehat{\mathcal{D}^1_\g}}

\newcommand{\Dg}{{\mathcal{D}^1_\g}}

\newcommand{\kndual}[2]{\langle #1,#2\rangle}
\newcommand{\ldot}{\,.\,}

\newcommand{\la}{\lambda}

\newcommand{\sbul}{{\boldsymbol{\cdot}}}
\newcommand{\Sp}{\mathcal{F}^{-1/2}}
\newcommand{\Sa}{\mathcal{S}}

\newcommand{\U}{\mathcal{U}}
\newcommand{\Uh}{\widehat{\mathcal{U}}}

\newcommand{\sln}{\mathfrak{sl}}

\newcommand{\gl}{\mathfrak{gl}}
\newcommand{\slnb}{{\overline{\mathfrak{sl}}}}

\newcommand{\ga}{\gamma}
\newcommand{\gah}{\widehat{\gamma}}
\newcommand{\gA}[1][{}]{\gamma^{\mathcal{A}}_{{#1}}}
\newcommand{\gS}[1][{}]{\gamma^{\mathcal{S}}_{{#1}}}
\newcommand{\gL}[1][{}]{\gamma^{\mathcal{L}}_{{#1}}}
\newcommand{\gm}[1][{}]{\gamma^{(m)}_{{#1}}}
\newcommand{\ggb}[1][{}]{\gamma^{\gb}_{{#1}}}
\newcommand{\gAh}[1][{}]{\widehat{\gamma}^{\mathcal{A}}_{{#1}}}

\newcommand{\gLh}[1][{}]{\widehat{\gamma}^{\mathcal{L}}_{{#1}}}
\newcommand{\gmh}[1][{}]{\widehat{\gamma}^{(m)}_{{#1}}}

\newcommand{\refR}[1]{Remark~\ref{R:#1}}

\renewcommand{\l}{\lambda}

\newcommand{\tr}{\mathrm{tr}}

\newcommand{\ddz} {\frac {d}{dz}}

\begin{document}

\title[$N$-point Virasoro algebras are multi-point KN type algebras]
{$N$-point Virasoro algebras 
 are 
multi-point Krichever--Novikov type algebras}


\author[Martin Schlichenmaier]{Martin Schlichenmaier}
\address{University of Luxembourg, 
Mathematics Research Unit, FSTC,
Campus Kirchberg, 6, rue Coudenhove-Kalergi,
L-1359 Luxembourg-Kirchberg, Luxembourg}
\curraddr{}
\email{martin.schlichenmaier@uni.lu}
\thanks{Partial  support by the 
Internal Research Project  GEOMQ11,  University of Luxembourg,
and 
by the OPEN scheme of the Fonds National de la Recherche
(FNR), Luxembourg,  project QUANTMOD O13/570706 
is gratefully acknowledged.}


\date{4.5.2015}

\begin{abstract}
We show how the recently again discussed $N$-point Witt, Virasoro,
and affine Lie
algebras are genus zero 
examples of the multi-point versions of Krichever--Novikov
type
algebras as introduced and studied by Schlichenmaier. Using this more
general point of view, useful structural insights and an easier access to
calculations
can be obtained. The concept of almost-grading will yield information
about
triangular decompositions 
which are of importance in the theory of representations. 
As examples the algebra of functions,
vector fields, differential operators, current algebras, affine Lie
algebras, Lie superalgebras
 and their
central extensions are studied. Very detailed calculations 
for the three-point case are given. 

\end{abstract}

\maketitle
\section{Introduction}

Recently there was again a revived interest in algebras of 
meromorphic objects (vector fields, Lie algebra valued functions, and
more)
on the Riemann sphere
\cite{CGLZ}, \cite{CJM}, \cite{JM},
\cite{Kreusch}, \cite{LeiMor}.
In particular, this interest comes from  representation theory and
its interpretations in the context of quantization of
(conformal) field theory.
The appearing 
algebras supply examples of infinite dimensional (Lie) algebras which
are of geometric origin. They generalize the Witt and Virasoro algebra 
respectively the classical affine Lie algebras.
In some of these articles  the vector field algebras were called
$N$-Virasoro algebras. 
Here we like to stress the fact, that these algebras are nothing
else as the genus zero Krichever--Novikov (KN) type algebras 
in their multi-point
version as introduced by the current author.

Originally KN algebras were introduced by Krichever and Novikov in
1986/87
\cite{KNFa}, 
\cite{KNFb}, 
\cite{KNFc} with the intention to generalize the classical 
infinite-dimensional algebras of Conformal Field Theory (CFT) to
higher genus. The classical algebras correspond to 
the geometric situation of genus zero and 
two fixed points where poles are allowed. 
In the original Krichever -- Novikov approach still only
two possible points for poles are considered.
In the years 1989/1990 the author extended the whole set-up to
the multi-point case for 
arbitrary genus (including genus zero) 
\cite{Schlmp},
\cite{Schleg},
\cite{Schlce},
\cite{SchlDiss}.

In this work I will show that the consequent use of the 
techniques developed in these articles  and the follow-ups
\cite{Schlloc}, \cite{Schlaff},   adapted to the
genus zero situation, will yield a much better understanding
of the situation. It will explain certain properties 
remarked by the authors of 
\cite{CGLZ}, \cite{CJM}, \cite{JM},
and remove some misconceptions.

The first nontrivial part in the multi-point extension is the fact 
that it is possible to assign to every non-empty splitting of
the set $A$ of points where poles are allowed an almost-graded
structure (see \refD{almgrad}). This is done 
by choosing a basis adapted to the
splitting.
An almost-grading is very close to a grading.
The notion is strong enough to construct triangular decompositions, 
semi-infinite wedge forms, Fock space representations, etc. 
These
concepts are of relevance in representation theory and
field theory. 
In the original Krichever--Novikov approach (and in the classical
case)
there is only one splitting possible. Hence, the additional 
effects due to different splittings do not show up.
They were studied the first time in \cite{Schlce}.
One of the consequences of the almost-grading will be that in
the structure equations (with respect to the adapted basis) 
the number of basis elements appearing 
on the r.h.s. will be bounded by an 
uniform constant. 

In the genus zero case with $N$ points $P_1,P_2,\ldots, P_N$ where
poles are allowed, one could single out one point, e.g. $P_N$ and move
it
by a fractional linear transformation to the point $\infty$.
Having done this we could take
\begin{equation}
\{P_1,P_2,\ldots, P_{N-1}\}\quad \cup\quad \{\infty\}
\end{equation}
as one possible splitting. We will denote this splitting
in this article {\it standard splitting}, despite the fact that
it will depend on which point was chosen to be moved to $\infty$.
Also we point out that for $N>2$ this is just one example of a
splitting.
Now we are in the situation that we could employ all the
constructions for the Krichever-Novikov type algebras 
carried out  by the author. 
For the convenience of the reader will recall them 
for the genus zero case. Furthermore, we will
strengthen 
the results if possible.
No previous knowledge of KN type algebras will be needed.
As a starting point 
we introduce the Poisson algebra of meromorphic forms.
From this algebra the associative algebra of functions, the Lie
algebra
of vector fields and differential operators, superalgebras, current
algebras, etc. are deduced.

Central extensions are of fundamental importance in the context
of regularization of actions.
Such regularizations are needed e.g. 
in quantum field theory.
As our presented algebras are assumed to act as symmetry algebras, the
classification of central extensions is a crucial task to be done.
Of course, one is interested to classify all central extensions up to
equivalence. But coming from the applications in representation theory
one also needs to know which central extensions admit an 
extension of the almost-grading of the original algebra.
Central extensions are defined via Lie algebras two-cocycles with 
values in the trivial module.
A complete classification of those
two-cycles classes which allow the extension of the almost-grading
(those cocycles are called {\it local}, see \refD{local}), 
and more general of the bounded ones for all the above Lie algebras
are given by the author in \cite{Schlloc}, \cite{Schlaff},
\cite{Schlsa}.
This is done always with respect to a fixed splitting and induced
almost-grading.
The cocycles are given by a differential to be 
integrated about special integration paths, respectively
by calculating residues.

\medskip

One of the main results of the current article is that for
genus zero each cocycle class is a bounded class
with respect to the standard splitting (\refT{tfclass}, \refT{vbound},
\refT{mbound}).
For the function algebra we need to add the natural requirement
for the cocycle to be $\La$-invariant or  equivalently
multiplicative.
We use our earlier results about bounded cocycles
to give them explicitely.
Furthermore, we show in genus zero that the vector field algebra and
the differential operator algebra are perfect. As we know their
cocycle classes we can write down the universal central extension. 
The dimension of the center for the vector field algebra is $N-1$;
for the differential operator algebra it is $3(N-1)$.
The existence of an universal central extension for the vector
field algebra for arbitrary genus is known \cite{Skry}, but here
we supply an elementary proof.
The algebra obtained  
as central extension from the function 
algebra via geometric cocycles (which might be  called
the {\it maximal Heisenberg algebra}), has also a $(N-1)$-dimensional 
center.
The current algebra associated to a finite-dimensional
simple Lie algebra $\g$ admits also an universal central
extension
which has an $(N-1)$-dimensional center \cite{Schlaff}, \cite{Breme}.
For 
the vector field algebra, current algebra and function
algebra 
there will be a unique (up to equivalence and
rescaling) central extension compatible with the almost-grading,
i.e. a local cocycle class.
For the differential operator algebra the space of local cohomology
classes will be three-dimensional.

\medskip
 
After having developed the general picture for genus zero and
$N$-points
we show how to do explicite calculations in the 3-point case.
{}From the side of applications this case is also of special interest, see e.g.
\cite{BeTer},  \cite{HaTe}.
We normalize the three points to be
$\{0,1\}\cup\{\infty\}$
and use additional symmetry operations.
All calculations  reduce to calculations
of residues of rational functions.
This is also the case for general $N$.

In \refS{further} we recall  some types of
representations which will be automatically available
after having identified the algebras as special cases of
KN type algebras.

In this article we mostly use the language of Riemann surfaces.
But all definitions, objects,  and  results clearly make sense
for arbitrary algebraically closed fields 
$\mathbb{K}$ of characteristics
zero. Our Riemann sphere will be the projective line over
$\mathbb{K}$. 

It should be possible to study this article without 
consulting the
works on KN type algebras mentioned above, assuming that the
reader is willing to accept the statements of the theory.
In case that the reader wants to know more we refer in addition
to the original work also to the recent monograph 
\cite{Schlkn} containing
everything what is needed.

\section{The Virasoro algebra and its relatives}
\label{sec:vira}

The
Virasoro algebra together with its relatives are 
the simplest non-trivial infinite
dimensional Lie algebras. 
As we will generalize them we recall their definitions here.

\subsection{The Witt algebra}

The {\it Witt algebra} $\W$ , sometimes also called 
Virasoro algebra without central
  term,
is the Lie algebra  
generated as vector space by the basis
elements 
$\{e_n\mid n\in \mathbb{Z}\}$ with 
Lie structure  
\begin{equation}\label{eq:Wstruct}
[e_n,e_m]=(m-n)\,e_{n+m},\quad n,m\in\Z.
\end{equation}

The algebra 
$\W$ is a graded Lie algebra.
We define the degree   by 
$\deg(e_n):=n$ 
and get the vector space direct decomposition
\begin{equation}
\W=\bigoplus_{n\in\Z}\W_n, \qquad 
\W_n={\langle e_n\rangle}_\C.
\end{equation}
Obviously,
 $\deg([e_n,e_m])=\deg (e_n)+\deg(e_m)$.

Algebraically $\W$ can also be given as
Lie algebra of derivations of the algebra of Laurent
polynomials $\C[z,z^{-1}]$.

\subsection{The Virasoro algebra}

For the Witt algebra the  universal 
central extension is the 
{\it Virasoro algebra} $\V$. 
As vector space it is the direct sum $\V=\C\oplus \W$.
We set for $x\in\W$, 
$\hat x:=(0,x)$, and $t:=(1,0)$.
Its basis elements are $\hat e_n, \ n\in\Z$ and 
$t$ with the  Lie product
\footnote{
Here $\delta_k^l$ is the Kronecker delta which is equal to 1 if
$k=l$, otherwise zero.}.
\begin{equation}\label{eq:Vstruct}
[\hat e_n,\hat e_m]=(m-n)\,\hat e_{n+m}+\frac 1{12}(n^3-n)\delta_m^{-n}\,t,\quad 
\quad
[\hat e_n,t]=[t,t]=0,
\end{equation}
for all $n,m\in\Z$. The factor 1/12 is conventional.
By setting
 $\deg(\hat e_n):=\deg(e_n)=n$ and $\deg(t):=0$  the Lie algebra 
$\V$ becomes a graded algebra. 
Up to equivalence of central extensions 
and rescaling the central element $t$, this is beside
the trivial (splitting) central extension the only central extension
of $\W$.

\subsection{The affine Lie algebra}

Given $\g$ a finite-dimensional Lie algebra (e.g. a 
finite-dimensional  simple
Lie algebra) then the tensor product of $\g$ with the associative
algebra of Laurent polynomials $\C[z,z^{-1}]$ carries a 
Lie algebra structure via
\begin{equation}
[x\otimes z^n,y\otimes z^m]:=[x,y]\otimes z^{n+m}, \quad
x,y\in\g\,.
\end{equation}
This algebra is called 
\emph{current algebra} or \emph{loop algebra} and denoted by $\gb$.
Again we consider central extensions. For this let $\beta$ be 
a symmetric, bilinear form for $\g$ which is invariant
(e.g. $\beta([x,y],z)=\beta(x,[y,z])$ for all $x,y,z\in\g$).
Then a central extension is given by
\begin{equation}
[\widehat{x\otimes z^n},\widehat{y\otimes z^m}]
:=\widehat{[x,y]\otimes z^{n+m}}-\beta(x,y)\cdot n\,
\delta_m^{-n}\cdot t.
\end{equation}
This algebra is denoted by $\gh$ and called 
\emph{affine Lie algebra}. 
With respect to the 
classification of Kac-Moody Lie
algebras, 
in the case of a simple $\g$
they are exactly the Kac-Moody
algebras of untwisted affine type, \cite{Kacorg}, \cite{KacB},  \cite{Moody69}.

\subsection{The geometric interpretation}
\label{sec:geo0}
Let $S^2$ be the Riemann sphere or  equivalently $\poc$
the projective line over $\C$. 
Denote by $z$ the quasi-global coordinate on $\poc$.
The above algebras and basis elements have a geometric meaning.
The elements are meromorphic objects
which are holomorphic outside $\{0,\infty\}$.
For example
the algebra $\C[z,z^{-1}]$
can be given as the algebra of meromorphic functions on
$S^2=\poc$ holomorphic outside of $\{0,\infty\}$.

The elements of the Witt algebra coincide with meromorphic vector fields, the
element of the current algebra with $\g$-valued meromorphic 
functions and so on.
The Lie algebra structure of $\W$ corresponds to the 
usual Lie bracket of vector fields 
\begin{equation}\label{eq:liebracket}
[v,u]=\left(v\ddz u-u\ddz v\right)\ddz.
\end{equation}
The basis elements are realized as 
$e_n=z^{n+1}\frac {d}{dz}$.

The central terms also have a geometric meaning which  will become
clear in the context of the generalization 
of these algebras described in the next sections.

\section{The general $N$ point and genus zero case}
\label{sec:npoint}
Even if in this article we are only interested in  genus zero 
it is necessary to spend a few words on the arbitrary
genus situation. In this context the full geometric
meaning will become clearer. 
After we have done this we will concentrate on the
genus zero case.

\subsection{Arbitrary genus and multi-point situation}
As explained in \refS{geo0} 
in the geometric interpretation for the Virasoro algebra the objects are
defined on the Riemann sphere and might have poles at most at two
fixed
points (which can be normalized to $0$ and $\infty$)
where poles are allowed. 
In applications, e.g. 
for a global operator approach to conformal field theory and
its quantization, integrable systems, etc.,
 this is not sufficient. One needs Riemann surfaces of
arbitrary genus. Moreover, one needs more than two points
where poles are allowed. 
Such  generalizations were initiated by Krichever and Novikov
\cite{KNFa}, \cite{KNFb}, \cite{KNFc}, who considered 
arbitrary genus and the two-point case. 
This was extended to the multi-point case (and this case is
of relevance here) and 
systematically examined by the author 
\cite{Schlmp}, 
\cite{Schleg}, 
\cite{Schlce}, 
\cite{SchlDiss}, 
\cite{SchlDeg},
\cite{SchlHab}, 
\cite{Schlloc}, 
\cite{Schlaff},
\cite{Schlkn}.

For the moment let $\Sigma_g$ be a compact Riemann surface 
without any restriction for its genus $g=g(\Sigma_g)$.
Furthermore, let $A$ be a finite subset of $\Sigma_g$.
Later we will need a splitting of $A$ into two non-empty disjoint
subsets $I$ and $O$, i.e. $A=I\cup O$. Set $N:=\#A\ge 2$,
$K:=\#I$, $M:=\#O$, with $N=K+M$. 
More precisely, let
\begin{equation}
I=(P_1,\ldots,P_K),\quad\text{and}\quad
O=(Q_1,\ldots,Q_{M})
\end{equation}
be disjoint  ordered tuples of  distinct points (``marked points'',
``punctures'') on the Riemann surface.
In particular, we assume $P_i\ne Q_j$ for every
pair $(i,j)$. The points in $I$ are
called the {\it in-points}, the points in $O$ the {\it out-points}.
Sometimes,  we consider $I$ and $O$ simply as sets.

Our algebraic objects will be objects meromorphic
on $\Sigma_g$ and holomorphic outside of $A$.
The corresponding algebras are called Krichever-Novikov type
algebras.
An almost-grading (see \refD{almgrad}) is introduced
with respect to the splitting of $A$.
The general theory can be found in the above references
and will not be repeated here. For the rest of
this contribution we will specialize the
theory for the genus zero
multi-point situation. Of course, for this the multi-point
theory developed by
the author plays an important role.

\subsection{Genus zero}
Now let $\Sigma_0$ be the Riemann sphere $S^2$, or equivalently 
$\poc$ with the quasi-global coordinate $z$.
We call it quasi-global, as  it is not defined
at the point $\infty$.
Let us denote the set of points 
\begin{equation}
A=\{P_1,P_2,\ldots, P_N\}, \quad P_i\ne P_j,\ \text{for } i\ne j.
\end{equation} 
For notational simplicity we single out the point 
$P_N$ as reference point. 
By an automorphism of $\poc$, i.e. a fractional linear
transformation or equivalently an element of
$\PGL(2,\C)$, the point $P_N$ can be brought to $\infty$. 
In fact two more points could be normalized to be $0$ and $1$.
In this section we will not do so, but see \refS{3point}.

Our points are given by their coordinates
\begin{equation}
P_i={a_i},\quad a_i\in\C, \ i=1,\ldots, N-1,
\qquad
P_N={\infty}.
\end{equation}
At these points we have the local coordinates
\begin{equation}
z-a_i,\ i=1,\ldots, N-1,
\qquad
w=1/z.
\end{equation}
Sometimes we refer to the classical situation. By this
we understand 
\begin{equation}\label{eq:class}
\Sigma_0=\Pro^1(\C)=S^2, \quad I=\{0\},\quad 
 O=\{\infty\}.
\end{equation}

\subsection{Meromorphic forms}

To introduce the elements of the algebras we first have to
consider  forms of (conformal) weight $\lambda\in 1/2\Z$.
Without further saying, 
we always assume that they are meromorphic and holomorphic
outside of $A$.

Forms of weight $0$ are functions. Of course they constitute
an associative algebra, which we denote by $\A$.
Forms of weight $1$ are (meromorphic) differentials.
Recall that the canonical line bundle $\K$ of $\Sigma$
is the holomorphic line bundle whose local sections are the 
the local holomorphic differentials.
For $\poc$,
in the language of algebraic geometry,  we have that 
$\K=\mathcal{O}(-2)$.
This bundle has a unique square root $L=\mathcal{O}(-1)$,
\footnote{This is not true anymore for Riemann surfaces
of higher genera $g$. In fact, we have $2^{2g}$ different square roots.
They correspond to different spin structures.} 
which is the tautological bundle, respectively the dual of the
hyperplane section bundle. We denote this bundle also by
$\K^{1/2}$.

Meromorphic forms of weight $\l$ are sections of the 
bundle $\K^{\l}$ where  (with $\K^*$ the dual bundle of $\K$)
(1) $\K^{\la}=\K^{\otimes \la}$ 
for $\la>0$, 
(2)  $\K^{0}=\mathcal{O}$,
 the trivial line bundle, and 
(3) $\K^{\la}=(\K^*)^{\otimes (-\la)}$ 
for $\la<0$.
We set
\begin{equation}
\qquad\Fl:=\{f \text{ is a global meromorphic section of } \K^\la
\mid
f \text{ is  holomorphic on } \Sigma\setminus A\}.\qquad
\end{equation}
Obviously this is an infinite dimensional $\C$-vector space.
Its elements  are called \emph{meromorphic forms
of weight $\la$}.
In particular, $\Fl[0]=\A$.
Forms of weight $-1$ are (meromorphic) vector fields
and we set $\La:=\Fl[-1]$.

In local coordinates $z_i$ 
a form of weight $\l$ can  be
written as $b_i(z)(dz_i)^\l$ with $b_i$ a meromorphic function.
By $(dz_i)^\l$ it is encoded how the local functions transform
under coordinate transformation.
In our genus zero situation this simplifies. We can describe the
form by a meromorphic function on the affine part $\C$ with respect
to the coordinate $z$. By this description its behaviour at
the point $\infty$ is uniquely fixed by 
the fundamental transformation $dz=-w^{-2}dw$.
Moreover, by  the fixing $P_N=\infty$
the set of meromorphic forms 
$f$ of weight $\l$ on $\poc$  holomorphic outside of 
$A$
correspond 1:1 to 
 meromorphic functions $a(z)$
holomorphic outside of $A$ via 
$f(z)=a(z)dz^\l$.
Both $a$ and $f$ will have the same
orders at the points in $\C$. For the order at the point 
$\infty$ we have 
\begin{equation}\label{eq:ordin}
\ord_{\infty}(f)=\ord_{\infty}(a)-2\la.
\end{equation}
Recall that 
on  a compact
Riemann surface 
the sum of the  orders (summed over all points)
of a meromorphic function $f\not\equiv 0$ 
equals zero.
Hence, more generally
\begin{proposition}\label{P:kldeg}
Let $f\in\Fl$, $f\not\equiv 0$ then
\begin{equation}
\sum_{P\in\Sigma_0}\ord_P(f)=-2\la.
\end{equation}
\end{proposition}
For this and related results see e.g. \cite{SchlRS}.
Also recall that the meromorphic functions in our case are
nothing else as rational functions with respect to the variable $z$.

Next we introduce algebraic operations on the vector space 
of meromorphic forms of arbitrary weights
(integer or  half-integer). 
We introduce the space
\begin{equation}
\mathcal{F}:=\bigoplus_{\la\in\frac 12\Z}\Fl,
\end{equation}
obtained  
by summing over all
weights.
The basic  operations will allow us to introduce 
finally the 
algebras we are heading for.
The following constructions make perfect sense in
the arbitrary genus case and the statement have been
proven there. For completeness we recall the results.

\subsection{Associative structure}
\label{sec:asstr}

The natural map of the 
locally free sheaves of rang one  
\begin{equation}\label{eq:afs}
\K^{\lambda}\times \K^\nu\to
 \K^{\lambda}\otimes \K^\nu\cong 
\K^{\lambda+\nu},
\quad
(s,t)\mapsto s\otimes t,
\end{equation}
defines a bilinear map 
\begin{equation}\label{eq:afl}
\sbul:\Fl\times \Fl[\nu]\to \Fl[\la+\nu].
\end{equation}
With respect to local trivialisations this corresponds to the
multiplication of the local representing meromorphic functions
\begin{equation}
(s\, dz^{\la},t\, dz^{\nu} )
\mapsto s\, dz^{\la}\;\sbul\;t\, dz^{\nu}= s\cdot t\;
dz^{\la+\nu}.
\end{equation}

\begin{proposition}\label{P:aass}
The space 
$\mathcal{F}$ is an  associative and commutative
graded (over $\frac 12\Z$)
algebra. Moreover, $\A=\Fl[0]$ is a subalgebra and the $\Fl$ are
modules over $\A$.
\end{proposition}

\subsection{Lie and Poisson algebra structure}
\label{sec:LiPof}

There is a Lie  algebra structure on the space $\mathcal{F}$.
The structure is  induced by the map 
\begin{equation}\label{eq:alies}
\Fl\times \Fl[\nu]\to \Fl[\la+\nu+1],
\qquad (e,f)\mapsto [e,f],
\end{equation}
which is defined in local representatives of the sections by
\begin{equation}\label{eq:aliea}
(e\, dz^{\la},f\, dz^{\nu} )
\mapsto [e\, dz^{\la},f\, dz^{\nu}]:= \left((-\la)e\pfz{f}+\nu
  f\pfz{e}
\right)
dz^{\la+\nu+1},
\end{equation}
and bilinearly extended to $\mathcal{F}$.
\begin{proposition}
\cite[Prop.~2.6 and~2.7]{Schlkn}
The prescription $[.,.]$ given by \refE{aliea} is
well-defined and
defines a Lie algebra structure 
on the vector space $\mathcal{F}$.
\end{proposition}

\begin{proposition}
\cite[Prop.~2.8]{Schlkn}
The subspace $\La=\Fl[-1]$ is a Lie subalgebra, and the 
$\Fl$'s are Lie modules over $\La$.
\end{proposition} 
\begin{theorem}
\cite[Thm.~2.10]{Schlkn}
The triple   $(\mathcal{F}\ ,\ \sbul\ ,\ [.,.])$ is a Poisson  algebra.
\end{theorem}

As substructures
we already encountered the subalgebras $\A$ 
of meromorphic functions and the subalgebra  $\La$ of 
meromorphic vector fields. The spaces $\Fl$ are
modules over them.

For the vector fields we obtain by the above   the usual 
Lie bracket and the usual Lie derivative for their actions on 
forms.
Written explicitly 
for the vector fields we get
\begin{equation}\label{eq:aLbrack}
[e,f]_|=[e(z)\fpz, f(z)\fpz]=
\left( e(z)\pfz f(z)- f(z)\pfz e(z)\right)\fpz \ ,
\end{equation}
for $e,f\in\La$. We used the same symbol 
for the vector field and for the representing function.
For the Lie derivative we get
\begin{equation}\label{eq:alied}
\nabla_e(f)_|=L_e(g)_|=
e\ldot g_{|}=
\left( e(z)\pfz f(z)+\lambda f(z)\pfz e(z)\right)\fpz\ .
\end{equation}

\subsection{The algebra of differential operators}
\label{sec:algdo}

The Lie algebra $\mathcal{F}$,
has $\Fl[0]$ as an abelian Lie subalgebra. The vector space sum
$\Do=\Fl[0]\oplus\Fl[-1]=\A\oplus \La$ 
is also a Lie subalgebra.
In an equivalent way this can also be constructed as 
semidirect sum of $\A$ considered as abelian Lie algebra
and $\La$ operating on $\A$ by taking the derivative.
The algebra  $\Do$
is the \emph{Lie algebra of differential operators of degree
$\le 1$}. 
In terms of elements the Lie product is 
\begin{equation}
[(g,e),(h,f)]=(e\ldot h-f\ldot g\,,\,[e,f]).
\end{equation} 
The projection on the second factor $(g,e)\mapsto e$ 
is a Lie homomorphism and we 
obtain a short exact sequences of Lie algebras
\begin{equation}
\begin{CD}
0 @>>> \A  @>>>\Do @>>> \La @>>> 0\ .
\end{CD}
\end{equation}
Hence $\A$ is an  (abelian) Lie ideal of $\Do$ and $\La$ a quotient 
Lie algebra.
Obviously, $\La$ is also a subalgebra of $\Do$.

The vector space $\Fl$ becomes a Lie module over $\Do$ 
by the operation
\begin{equation}
(g,e).f:=g\cdot f+e.f,\quad
(g,e)\in\Do(A), f\in\Fl(A).
\end{equation}

\medskip

Via some universal constructions differential operator
algebras of arbitrary degrees can be constructed.
We will not repeat their definition here, but only refer to 
\cite[Chap.~2.7]{Schlkn}


\subsection{Current algebras}
\label{sec:currint}
We fix an arbitrary finite-dimensional complex Lie algebra $\g$.
The generalized current algebra
is defined as 
$\gb=\g\otimes_\C \A$ with the Lie product
\begin{equation}\label{eq:currint}
[x\otimes f, y\otimes g]=[x,y]\otimes f\cdot g,
\qquad  x,y\in\g,\quad f,g\in\A.
\end{equation} 
It can be easily verified that 
$\gb$ is a Lie algebra.

Later we will introduce central extensions for
these current algebras. They will generalize affine Lie algebras,
respectively affine Kac-Moody algebras of untwisted type.

\subsection{$\g$-differential operators}
\label{sec:gdiff}
For some applications (e.g. for the fermionic 
Fock space representations, for the Sugawara representation) 
it is useful to extend the definition 
of the current algebras by
considering differential operators (of degree $\le 1$) associated 
to $\gb$. We define 
$\Do_\g:=\gb\oplus\La$ and take in the summands the Lie products
defined there and put additionally
\begin{equation}
[e,x\otimes g]:=-[x\otimes g,e]:=x\otimes(e.g).
\end{equation}
This operation can be described as semidirect sum
of $\gb$ with $\La$ and we get 
\begin{proposition}
\cite[Prop.~2.15]{Schlkn}
$\Do_\g$ is  a Lie algebra.
\end{proposition}

All the above algebras I call Krichever--Novikov type
algebras, despite the fact that 
Krichever and Novikov did not introduce all these types of algebras.
Furthermore, they 
only considered
the two-point case.


\section{Choice of a basis and an almost-grading}
\label{sec:almost}

\subsection{Definition of an almost-grading}
In the classical situation 
the introduced algebras 
are graded algebras. 
In the higher genus case and even in the genus zero case with more
than
two points where poles are allowed there is no non-trivial grading
anymore.
As realized by Krichever and Novikov \cite{KNFa}
there is a weaker concept, an almost-grading, which
to  a large extend is a valuable replacement of a honest grading.
Such an almost-grading is induced by a splitting of the set $A$ into
two non-empty and disjoint sets $I$ and $O$. 
The (almost-)grading is fixed by exhibiting 
a certain basis 
of the spaces $\Fl$ and define  these elements 
to be (almost-) homogeneous.
\begin{definition}\label{D:almgrad} 
Let $\La$ be a Lie or an associative algebra such that
\begin{equation}
\La=\oplus _{n\in\Z}\La_n
\end{equation}
is a vector space direct sum, then
$\La$ is called an \emph{almost-graded} 
(Lie-) algebra if
\begin{enumerate}
\item[(i)] $\dim \La_n<\infty$,
\item[(ii)]
There exists constants $L_1,L_2\in\Z$ such that 
\begin{equation*}
\La_n\cdot \La_m\subseteq \bigoplus_{h=n+m-L_1}^{n+m+L_2}
\La_h,\qquad \forall n,m\in\Z.
\end{equation*}
\end{enumerate}
The elements in $\La_n$ are called {\it homogeneous} elements of
degree $n$, and $\La_n$ is called \emph{homogeneous subspace}
 of degree $n$.

If $\dim\La_n$ is bounded with a bound independent of $n$ we call $\La$ 
\emph{strongly almost-graded}. If we drop the condition that
$\dim\La_n$ is finite we call $\La$ \emph{weakly almost-graded}.
\end{definition}

In a similar manner almost-graded modules over almost-graded algebras
are defined.
Furthermore, 
this definition makes complete sense also for more general index sets
$\J$. 

\subsection{Existence of an almost-grading}

Given a splitting $A=I\cup O$,
one of the results of the author is that all the
above introduced algebras $\A,\La,\Do$, $\gb$, and $\Dg$ 
admit a strongly almost-graded structure induced by a well-defined
procedure given by fixing an adapted basis.
Essentially different splittings (meaning not just obtained by
inverting the role of $I$ and $O$) will yield ``non-equivalent''
almost-gradings.
\begin{remark}
In the two-point case, i.e. the classical case, or more generally
the case considered by Krichever and Novikov, the set $A$ consists of
two points. Hence, there is only one splitting possible and 
consequently
only one almost-grading. This is not the case
anymore for more than two points.
To realize the importance of the splitting was
a crucial observation by the author
\cite{SchlDiss}, \cite{Schlce}.
\end{remark}

We set $\Jl=\Z$ for $\l\in\Z$ and  $\Jl=\Z+1/2$ for 
$\l\in\Z+1/2$.
Given a splitting with $\#I=K$ the author gives a 
procedure to exhibit a certain basis of $\Fl$ 
\begin{equation}
\{f_{n,p}^\la\mid n\in\Jl,\  p=1,\ldots, K\}
\end{equation}
with special properties. The
subspace
\begin{equation}
\Fln:={\langle f_{n,p}^\la\mid   p=1,\ldots, K\rangle}_{\C}
\quad\ \subset\ \Fl\, \qquad n\in\Jl
\end{equation} 
is the subspace
of homogeneous
elements of degree $n$. Then
\begin{equation}
\Fl=\bigoplus_{n\in\Jl}\Fln,\quad \dim \Fln=K.
\end{equation}
As we will give in the following an explicit 
construction  we will not need details of the
general theory.
We refer
 the interested reader
e.g. to \cite{Schlmp}, \cite{Schlce}, \cite{Schlkn} for details.

Let us numerate in our genus zero situation the points in the
splitting like
\begin{equation}
A=I\cup O, \qquad I=(P_1,P_2,\ldots, P_K),\quad
O=(P_{K+1},\ldots, P_N=\infty).
\end{equation} 
As $P_N=\infty\in O$ it is enough
to construct a basis 
$\{A_{n,p}\mid n\in\Z, p=1,\ldots K\}$
for $\A=\Fl[0]$. The decomposition of
$\A$ induces a decomposition of $\Fl$ by
\begin{equation}\label{eq:af}
\Fln=\A_{n-\la}dz^{\la},
\quad\text{respectively}
\quad
f_{n,p}^\la=A_{n-\la,p}dz^\la.
\end{equation}
The shift by $-\la$ is quite convenient and is beside other things
related to the  duality property discussed further down.
The recipe for constructing  the $A_{n,p}$ is given in 
\cite{Schlce}, \cite{Schleg}, see also \cite{Schlkn}.
As a principal property we have 
\begin{equation}
\ord_{P_i}(A_{n,p})=(n+1)-\delta_i^p,\ i=1,\ldots ,K\,.
\end{equation}
At the points in $O$ 
corresponding orders are set to make the element
unique up to multiplication by a non-zero scalar.
\begin{example}
We call the splitting
\begin{equation}\label{eq:standsplit}
  I=(P_1,P2,\ldots, P_K),\quad
O=(P_N=\infty), \quad K=N-1
\end{equation}
the \emph{standard splitting}. Recall that $P_i$ corresponds to the
point given by the coordinate $z=a_i$.
We set
\begin{equation}
\alpha(p):=\big(\prod_{\substack{i=1\\ i\ne p}}^K(a_p-a_i)\big)^{-1}.
\end{equation}
and define
\begin{equation}
A_{n,p}(z):=(z-a_p)^n\cdot
\prod_{\substack{i=1\\ i\ne p}}^K(z-a_i)^{n+1}
\cdot \alpha(p)^{n+1}\,.
\end{equation}
The last factor is a normalization factor yielding
\begin{equation}
A_{n,p}(z)=(z-a_p)^n(1+O(z-a_p)).
\end{equation}
With this description the order at $\infty$ is fixed as
\begin{equation}
\label{eq:inford}
-(Kn+K-1)\,.
\end{equation}
By \refE{af}  the basis elements
for the other $\Fl$ are given too.
In particular, we obtain for 
$\La$
the basis
\begin{equation}
e_{n,p}=f_{n,p}^{-1}=A_{n+1,p}\ddz
=
(z-a_p)^{n+1}\cdot\prod_{\substack{i=1\\ i\ne p}}^K(z-a_i)^{n+2}
\cdot \alpha(p)^{n+2}\,\ddz.
\end{equation}
\end{example}
\subsection{Duality}
\label{sec:dual}
The pairing which we describe here is valid for arbitrary
genus and arbitrary splittings.
Let $C_i$ be  positively oriented (deformed) circles around
the points $P_i$ in $I$, $i=1,\ldots, K$ 
 and $C_j^*$ positively oriented circles around
the points $Q_j$ in $O$, $j=1,\ldots, M$.
A cycle $C_S$ is called a {\it separating cycle} 
if it is smooth, positively oriented of multiplicity one and if 
it separates
the in-points  from the out-points. It might have multiple components. 
In the following we will integrate  
meromorphic differentials on $\Sigma_g$ without poles in 
$\Sigma_g\setminus A$ over closed curves $C$. 
Hence, we might consider  the $C$ and $C'$ as
equivalent if $[C]=[C']$ in  $\Ho_1(\Sigma_g\setminus A,\Z)$.
In this sense
we can write 
for every separating cycle $C_S$
\begin{equation}\label{eq:cs}
[C_S]=\sum_{i=1}^K[C_i]=-\sum_{j=1}^M [C^*_j].
\end{equation}
The minus sign appears due to the opposite orientation.
Another way for giving such a $C_S$ is  via
level lines of a ``proper time
evolution'', for which I refer to \cite[Section~3.9]{Schlkn}.

Given such a separating cycle $C_S$ (respectively cycle class) 
we define a linear map
\begin{equation}
\Fl[1]\to\C,\qquad \omega\mapsto \cins \omega.
\end{equation}
The map will not depend on 
the separating line $C_S$ chosen, as
two of such will be homologous and the poles of $\omega$ are only located in
$I$ and $O$.

Consequently,
the integration of $\omega$ over $C_S$ can 
also be described
over the 
special cycles $C_i$ or equivalently over $C_j^*$.
This integration 
corresponds to calculating residues
\begin{equation}\label{eq:res}
\omega\quad\mapsto\quad
\cins \omega\ =\ \sum_{i=1}^K\res_{P_i}(\omega)
\ =\ -\sum_{l=1}^{M}\res_{Q_l}(\omega).
\end{equation}
The pairing 
\begin{equation}\label{eq:kndual}
\Fl\times\Fl[1-\la]\to \C,\quad
(f,g)\mapsto\kndual{f}{g}:=\cins f\cdot g,
\end{equation}
between $\la$ and $1-\la$ forms is called
{\it Krichever-Novikov (KN) pairing}.

With respect to this pairing we have
\cite[Thm.~3.6]{Schlkn}
\begin{equation}\label{eq:knd}
\kndual{f_{n,p}^\la}{f_{-m,r}^{1-\la}}=
\delta_p^r\;\delta_n^m, \qquad \forall n,m\in\Jl,\quad 
r,p=1,\ldots, K.
\end{equation}
In particular, the pairing is 
 non-degenerate. 
This pairing is extremely helpful as
e.g. given $f\in\Fl$ then the expansion in terms of the
basis
\begin{equation}\label{eq:exp}
f=\sum_{n\in\Jl}\sum_{p=1}^K\alpha_{n,p}^\l f_{n,p}^\la,
\qquad \alpha_{n,p}\in\C
\end{equation}
can be determined via
\begin{equation}\label{eq:expc}
\alpha_{n,p}=
\kndual{f}{f_{-n,p}^{1-\la}}
=\sum_{P\in I}\res_P(f\cdot f_{-n,p}^{1-\la})
=-\sum_{Q\in O}\res_Q(f\cdot f_{-n,p}^{1-\la}).
\end{equation}
Note that the pairing depends not only  on $A$ (as the $\Fl$ depend on
it) but also critically on the splitting of $A$ into $I$ and $O$
as the integration path $C_S$ will depend on it. Once the splitting
is fixed the pairing will be fixed too.

\subsection{Almost-graded structure of the algebras}
{}From the general theory it follows that all our algebras
$\A,\La,\Do,\gb,\Dg$ are strongly almost-graded and the $\Fl$ are almost-graded
modules
over the first three.
More precisely, there exists $R_1$ and $R_2$ such that
\begin{equation}
\begin{aligned}{}
\A_n\cdot \A_m&\subseteq \bigoplus_{k=n+m}^{n+m+R_1}\A_k
\\
[\La_n, \La_m]&\subseteq \bigoplus_{k=n+m}^{n+m+R_2}\La_k.
\end{aligned}
\end{equation}
Similar expressions are there for the modules with exactly the
same bounds. 
The lowest order terms can be given as
\begin{equation}
\begin{aligned}[]
A_{n,p}\cdot A_{m,r}\ &=\ A_{n+m,r}\,\delta_r^p\ +\ \text{h.d.t.}
\\
A_{n,p}\cdot f_{m,r}^\la\ &=\ f_{n+m,r}^\la\,\delta_r^p\ +\ \text{h.d.t.}
\\
[e_{n,p},e_{m,r}]\ &=\ (m-n)\cdot e_{n+m,r}\,\delta_r^p\ +\ \text{h.d.t.}
\\
e_{n,p}\ldot f_{m,r}^\la\ &= \ (m+\lambda n)\
\cdot f_{n+m,r}^\la\,\delta_r^p\ +\ \text{h.d.t.}
\end{aligned}
\end{equation}
Here $\ \text{h.d.t.}\ $ denote linear combinations of 
basis elements of degree between
$n+m+1$ and $n+m+R_i$,

The other coefficients can be  explicitely  calculated.
Also it is easy to determine the upper bounds $R_1$ and $R_2$.
This is done by calculating the algebraic result of
the basis elements and then expanding the 
result via \refE{exp} and calculate by \refE{expc} the coefficients
via residues of rational functions.
As an example we consider $\A$. The structure equations of $\A$ are
given via
\begin{equation}\label{eq:aprodsum}
A_{n,p}\cdot
A_{m,r}=\sum_{h\in\Z}\sum_{s=1}^K\alpha_{(n,p)(m,s)}^{(h,s)}\,A_{h,s}.
\end{equation}
Recalling  that $f_{-h,s}^1=A_{(-h-1,s)}dz$ we set
\begin{equation}\label{eq:factor}
\omega=
A_{n,p}\cdot 
A_{m,r}\cdot A_{(-h-1,s)}dz.
\end{equation}
Then 
\begin{equation}
\alpha_{(n,p)(m,s)}^{(h,s)}
=\sum_{P\in I}\res_{P}(\omega)
= -
\sum_{Q\in O}
\res_{Q}(\omega).
\end{equation}
It is convenient to set $h=n+m+k$. 
By summing up all the orders of the factors 
of $\omega$ individually at the point $P\in I$ 
 we see that there is no residue at $I$ if $k<0$.
Hence the coefficients are vanishing in the sum  
\refE{aprodsum} for $h< n+m$.
Doing the same for the orders at the points $Q$ in $O$ 
we see that
there is a bound $R_1$ such that if $k>R_1$ 
there will be no residue at $O$. Hence we get the almost-grading.
The coefficients can be explicitely calculated by calculating residues
of rational functions.
The lowest term will only show up if $p=r=s$ and there is only
a residue at $P_p$ which is equal to $1$.

Exactly the same kind of arguments work for the algebra $\L$ 
were we now obtain the bound $R_2$. The same is true 
for the modules $\Fl$.

\subsection{Triangular decomposition}
Let $\U$ be one of the above introduced algebras (including the
current algebra). 
On the basis of the almost-grading we obtain a triangular
decomposition of the algebras
\begin{equation} 
\U=\U_{[+]}\oplus\U_{[0]}\oplus\U_{[-]},
\end{equation}
where 
\begin{equation}
\U_{[+]}:=\bigoplus_{m>0}\U_m,\quad
\U_{[0]}=\bigoplus_{m=-R_i}^{m=0}\U_m,\quad
\U_{[-]}:=\bigoplus_{m<-R_i}\U_m.
\end{equation}
By the almost-gradedness the $[+]$ and
$[-]$ subspaces are  
(infinite dimensional) subalgebras. The $\U_{[0]}$ are only 
finite-dimensional vector spaces.

Depending on the applications it is sometimes useful to enlarge 
the algebra $\U_{[-]}$ by adding a finite-dimensional subspace from 
$\U_{[0]}$ so that the enlarged algebra $\U_{[-]}^*$
contains all objects regular at the points in $O$, respectively
vanishing there with a certain order.

The existence of such a triangular decomposition 
indicates the importance of the existence of an almost-grading.
Such a triangular decomposition is necessary for 
developing a rich representation
theory.
On its basis one constructs highest weight representations,
Verma modules, Fock space representations and many more.
The elements of $\U_{[+]}$
quite often correspond to annihilation operators, the others to
creation operators. 

\subsection{Standard splitting}
For the standard splitting the
set $O$ consists only of the point $\infty$.
The elements 
$A_{n,p}$ for $p=1,\ldots,K=N-1$ are introduced above.
For illustration we give the bounds $R_1$ and $R_2$
\begin{proposition}
\begin{equation}
R_1=\begin{cases} 0,&N=2,
\\ 1,&N>2,
\end{cases}
\qquad \qquad
R_2=\begin{cases} 0,&N=2,
\\ 1,&N=3,
\\ 2,&N> 3\,.
\end{cases}
\end{equation}
\end{proposition}
\begin{proof}
For calculating the order at $\infty$ with respect to the
variable $w$ we use for the individual factors 
in the expression \refE{factor} 
the value \refE{inford} and sum over all factors and do not
forget to decrease the order by 2 coming from $dz$.
If we do this for the algebra $\A$
we get as order for 
$A_{n,p}\cdot A_{m,r}\cdot A_{-(n+m+k)-1}$ the value
$-2K+K\cdot k+1$. A pole is only possible if 
this value is $\le -1$. Hence only for
\begin{equation}
k\le -\frac {2}{K}+2\,.
\end{equation} 
This yields the claimed value for $R_1$.
For the Lie algebra $\La$, respectively for the Lie module we have
to consider
$A_{n+1,p}\cdot A_{m-\la,r}\cdot A_{-(n+m+k)-(1-\la)}$.
For the order at $\infty$ we
obtain
$-3K+K\cdot k+2$. Which yields that a pole is only
possible for
\begin{equation}
k\le -\frac {3}{K}+3\,,
\end{equation}
and hence the claimed value for $R_2$.
\end{proof}
The structure coefficients of the algebras
can be directly calculated by calculating
residues of rational functions via \refE{expc}.
We will not do it here for the general case.
In \refS{3point} we treat the three-point case in detail.
\begin{example}
{\bf $N=2$.}
By a $\PGL(2,\C)$ action the two points can be transported to
$0$ and $\infty$. This is the classical situation and there is 
only one splitting. Hence, everything is fixed. The above basis gives
back
the conventional one.
\end{example}
\begin{example}
{\bf $N=3$.}
Here by a $\PGL(2,\C)$ action the three points can be normalized to
$\{0,1,\infty\}$. Hence, up to isomorphy there are only one
three-point algebra (for each type).
If we fix such an algebra we obtain three different
splittings of the set $\{0,1,\infty\}$ and consequently
also 3 essentially different almost-gradings, triangular
decompositions, etc.
The three-point case is in a certain sense special as there are
still the automorphism of $\poc$ permuting these three points.
They induce automorphisms of the algebras which 
permute the almost-gradings.
We will consider this situation in detail in \refS{3point}.
\end{example}
\begin{example}
{\bf $N=4$.}
This is the first case where we have a moduli parameter 
for the geometric situation.
We  normalize our $A$ to 
\begin{equation}
\{0,1,a,\infty\},\qquad  a\in \C,\ a\ne 0,1\, .
\end{equation}
We have 2 different types of splittings, i.e. the type $4=3+1$ and 
the type $4=2+2$.
For example
\begin{equation}
\{0,1,a\}\cup \{\infty\},\quad\text{and}\quad
\{0,1\}\cup \{a,\infty\}\,.
\end{equation}
The first type is the standard splitting for which we gave the 
basis above.
For the second splitting a basis of $\A$ and hence of all $\Fl$ is
\begin{equation}
\begin{aligned}{}
A_{n,1}(z)&=z^n(z-1)^{n+1}(z-a)^{-(n+1)}a^{(n+1)},
\\
A_{n,2}(z)&=z^{n+1}(z-1)^{n}(z-a)^{-(n+1)}(1-a)^{(n+1)}\,,
\end{aligned}
\end{equation}
where $n\in\Z$.
The last factor is again a normalization constant.
This basis defines an almost-grading for the four-point algebra
which is not equivalent to the standard almost-grading. 
Again upper bounds for the level of the algebra coefficients
and the coefficients itself can be calculated easily via residues.
\end{example}

\subsection{Another  basis}
\label{sec:unhandy}
Clearly, our algebra $\A$ can be given as the algebra
\begin{equation}
\A=\C[(z-a_1),(z-a_1)^{-1},(z-a_2)^{-1},\dots, (z-a_{N-1})^{-1}],
\end{equation}
with the obvious relations.

If we introduce
\begin{equation}
A_n^{(i)}:=(z-a_i)^n,
\end{equation}
then 
\begin{equation}
A_n^{(i)},\quad n\in\Z,\ i=1,\ldots, N-1
\end{equation}
is a generating set of $\A$. A basis is given e.g. by
\begin{equation}\label{eq:anotherbas}
A_n^{(1)},\ n\in\Z,\quad
A_{-n}^{(i)},\ n\in\N,\ i=2,\ldots, N-1.
\end{equation}
\begin{lemma}
\label{L:agen}
Let $\A_{(0)}$ be the subalgebra of meromorphic functions holomorphic
outside of $\infty$ then 
\begin{equation}
\A_{(0)}={\langle A_n^{(1)},\ n\ge 0\rangle}_{\C}
={\langle A_n^{(2)},\ n\ge 0\rangle}_{\C}=
\cdots=
{\langle A_n^{(N-1)},\ n\ge 0\rangle}.
\end{equation}
\end{lemma}
Explicit calculations with respect to the
 basis \refE{anotherbas} have been done in 
\cite{Schlmp},
\cite{SchlDiss}.
Similar calculations were done e.g. by Dick 
\cite{Di}, Anzaldo-Meneses
\cite{Anz} and Guo, Na, Shen, Wang, Yu \cite{GuNaSh}.
It turns out that for the 
``products'' of certain type of elements  of this basis 
the number of elements in the results do not have a bound.
Hence, it is not possible to introduce a strong almost-grading of 
$\A$ such that these basis elements are homogeneous and  it is not
possible to construct triangular decompositions with respect 
to this basis.
After realizing this I switched in \cite{Schlce} to the 
kind of basis presented above.
Nevertheless this kind of basis will play a role in
some proofs later.
Of course, as above, a basis of $\A$ will yield a basis of
$\La$ (and $\Fl$).
We will also use
\begin{equation}
e_n^{(i)}=A_{n+1}^{(i)}\ddz,\qquad
n\in\Z,\  i=1,\ldots, N-1.
\end{equation}

\section{Central extensions}
\label{sec:central}

Central extension of our algebras appear naturally in the context of
quantization and the regularization of actions. 
 Of course, they are also of independent 
mathematical interest.

\subsection{Central extensions and cocycles}
For the convenience of the reader let us repeat the relation between
central extensions and the second 
Lie algebra cohomology with values in the trivial module.
A central extension of a Lie algebra $\U$ is a special
Lie algebra structure on the 
vector space direct sum $\widehat{\U}=\C\oplus \U$.
If we denote  $\hat x:=(0,x)$ and $t:=(1,0)$ then the Lie structure is
given by 
\begin{equation}\label{eq:cext}
[\hat x, \hat y]=\widehat{[x,y]}+\psi(x,y)\cdot t,
\quad [t,\widehat{\U}]=0,
\quad x,y\in \U\,,
\end{equation}
with bilinear form $\psi$.
The map  $x\mapsto \hat x=(0,x)$ is a linear splitting map.
$\widehat{\U}$ will be a Lie algebra, e.g. will fulfill the 
Jacobi identity, if and only if $\psi$ is an antisymmetric 
bilinear form and fulfills
the Lie algebra 2-cocycle condition
\begin{equation}\label{eq:2coc}
0=d_2\psi(x,y,z):=
\psi([x,y],z)+
\psi([y,z],x)+
\psi([z,x],y).
\end{equation}
A 2-cochain $\psi$ is a coboundary
if  there exists 
a linear form $\varphi:\U\to\C$ such that 
\begin{equation}\label{eq:2cob}
\psi(x,y)=\varphi([x,y]).
\end{equation}
Every coboundary is a cocycle. 
The second Lie algebra cohomology $\Ho^2(\U,\C)$ 
of $\,\U\,$ with 
values in the trivial module $\C$ is defined as the 
quotient of the space of 2-cocycles modulo coboundaries.

Two central extensions are equivalent if they essentially
differ by the choice of the  splitting maps.
They are equivalent if and only if the difference of 
their defining
2-cocycles $\psi$ and $\psi'$ is a coboundary.
In this way the second Lie algebra cohomology $\Ho^2(\U,\C)$ 
classifies equivalence classes
of central extensions. The class $[0]$ corresponds to the
trivial central extension. In this case the splitting map
is a Lie homomorphism. 
To construct central extensions of our algebras we have to 
find such Lie algebra 2-cocycles. 

Clearly, equivalent central extensions are isomorphic.
The opposite is not true. 
In our case we can always rescale
the central element by multiplying it with a nonzero scalar. 
This is an isomorphism but not an equivalence of
central extensions. Nevertheless, it is an irrelevant modification.
Hence we will be mainly interested in central extensions modulo 
equivalence and rescaling. They are classified by 
$[0]$ and the elements of the projectivized cohomology 
space $\Pro( \Ho^2(\U,\C))$.

Recall that if $\U$ is a perfect Lie algebra, i.e. if $[\U,\U]=\U$ 
then there exists a universal central extension and 
$k=\dim\Ho^2(\U,\C)$ gives the dimension of the center of this
extension.
Moreover, if this space is finite-dimensional then 
the universal central extension is 
up to equivalence given as follows.
As
vector space it is the direct
sum 
$\widehat{\U}_{univ}=\C^k\oplus\U$. Let $[\psi_i]$, $i=1,\ldots, k$ be a
basis
of $\Ho^2(\U,\C)$ and each class represented by a cocycle
$\psi_i$. Moreover, let $t_1,t_2,\ldots, t_k$ be 
standard basis elements of $\C^k$, then the Lie structure is given as
\begin{equation}
[\widehat{x},\widehat{y}]
=\widehat{[x,y]}+\sum_{i=1}^k\alpha_i\psi_i(x,y)\cdot t_i,
\quad x,y\in\U,\ \alpha_i\in\C \quad
[t_i,\widehat{U}_{univ}]=0\,.
\end{equation}

\subsection{Almost-graded central extensions}
Before we discuss for each of our algebras the central 
extensions separately we will treat their common features.
Denote by $\U$ one of these algebras.
Our algebras are almost-graded. Coming from the applications one is
quite often only interested in central extensions $\Uh$ which allow to 
extend the almost-grading of $\U$. This says that only those cocycles 
are allowed such that it is possible to assign to the 
central element $t$  a fixed degree such that
\begin{equation}
[\Uh_n,\Uh_m]\quad\subseteq\quad \sum_{h=n+m-L_1}^{n+m+L_2}\, \Uh_h,
\end{equation}
with $L_1$ and $L_2$ independent of $n$ and $m$.
If there is such a value for the degree of $t$, the value $\deg t=0$
will also do. 
Hence, without restriction we will take this value.
For $\hat x=(x,0)$ and $t=(0,1)$ we set
\begin{equation}
\deg \hat x=\deg x,\quad \deg t=0.
\end{equation}

\begin{definition}\label{D:local}
(a) Let $\ga$ be a 2-cocycle
for the almost-graded Lie algebra $\U$, then $\ga$ is called a
\emph{local
cocycle} if $\exists T_1,T_2$ such that
\begin{equation}
\ga(\U_n,\U_m)\ne 0\implies 
T_2\le n+m\le T_1.
\end{equation}
(b) A 2-cocycle  $\ga$ is called \emph{bounded} (from above) if 
 $\exists T_1$ such that 
\begin{equation}
\ga(\U_n,\U_m)\ne 0\implies 
n+m\le T_1.
\end{equation}
(c)  A cocycle class $[\ga]$ is called a 
\emph{local (bounded) cohomology  class} if and only if 
it admits a representing cocycle which is 
local (respectively bounded).
\end{definition}
Note that e.g. in a  local 
cocycle class not all representing cocycles are local.
Obviously, the set of local (or bounded) cocycles is 
a subspace of all cocycles. Moreover, the set 
$\Ho^2_{loc}(\U,\C)$  (respectively $\Ho^2_{b}(\U,\C)$)  
of
local (respectively bounded) cohomology  classes 
is a subspace of the full cohomology space.

We point out that what is local (and bounded) depends on the almost-grading
induced by the splitting $A=I\cup O$.
\begin{remark}
We could also introduce the space of 
$\widetilde{\mathrm{H}}_{loc}^2(\U,\C)$ of
local cocycles  modulo local coboundaries. 
This space can  naturally be 
identified with $\Ho^2_{loc}(\U,\C)$ as if two 
local cocycles $\ga_1$ and $\ga_2$ are cohomologous then the
corresponding coboundary is local too.
\end{remark}
\begin{remark}
In the classical Witt algebra case 
$\Ho^2(\W,\C)=\Ho^2_{loc}(\W,\C)$.
The corresponding result is neither true for higher genus, nor for
the multi-point situation.
\end{remark}

As explained above the almost-grading is crucial for the
triangular decomposition. Hence it should not be a surprise
if the cocycles obtained via regularisation processes from
the usual representations are local, see  \cite{SchlDiss}, \cite{Schlkn}.

\subsection{Geometric cocycles}
Our algebras $\U$ consists of geometric objects. Hence, it is 
quite natural to have a closer look at Lie algebra cocycles which
can be defined via geometric means.
\begin{definition}
A cocycle $\ga:\U\times\U\to\C$ is called a 
\emph{geometric cocycle} if there is a bilinear map
\begin{equation}
\gah:\U\times\U\to\Fl[1],
\end{equation} 
such that $\ga$ is the composition of $\gah$ with an integration, i.e.
\begin{equation}\label{eq:cyclgeo}
\ga=\ga_C:=\cint{C} \gah
\end{equation}
with $C$ a curve on $\Sigma_g$. 
A cohomology class $[\ga]$ is a 
\emph{geometric cohomology class} if
it contains a representing element $\ga'\in[\ga]$ 
which is a geometric cocycle.
\end{definition} 
\begin{proposition}
For a geometric cocycle 
$\ga_C=\ga_{C'}$ if 
$[C]=[C']\in\Ho_1(\Sigma_g\setminus A,\C)$.
\end{proposition}
\begin{proof}
As $\Fl[1]$ is the space of meromorphic differentials holomorphic
on $\Sigma_g\setminus A$ the integral over 
two cohomologous cycles will be the same.
\end{proof}
On this basis we can write 
$\ga_{[C]}$ and even allow that $[C]$ is an arbitrary
element from $\Ho_1(\Sigma_g\setminus A,\C)$.

In the opposite direction
after fixing a bilinear map $\gah$ the space  
$\Ho_1(\Sigma_g\setminus A,\C)$ will parameterize
(in a possibly non-unique way) the space of associated
cocycles $\ga_{[C]}$.
This is exactly the way which we will go to obtain cocycles.
\begin{remark}
That this works we first have  to verify that the
$\gah$ which we will choose is indeed a
differential, and then that the $\ga_C$  fulfills the 
Lie algebra cocycle conditions.
Also note that we do not claim that the parameterization is 1:1.
\end{remark}
It is well-known that 
\begin{equation}
\dim\Ho_1(\Sigma_g\setminus A,\C)=
\begin{cases} 2g,&\# A=0,1,
\\
2g+(N-1),&\# A=N\ge 2\,.
\end{cases}
\end{equation}
Generators for this vector space are given by the $2g$ standard
symplectic cycles and the cycles given by ``circles'' around
the points $P_i\in A$.
If $N\ge 1$ there is exactly one relation between these generators.
By the splitting $A=I\cup O$ we fix a separating 
cycle class $[C_S]$. This is a non-vanishing element of 
$\Ho_1(\Sigma_g\setminus A,\C)$. It will be a preferable element
to be taken as one of the basis elements as it will yield
for our cocycles the fact that it will be local (see below).

\medskip
For genus zero and $N\ge 1$ we have
\begin{equation}
\dim\Ho_1(\Sigma_0\setminus A,\C)=
(N-1). 
\end{equation}
A basis of the space is given by circles $C_i$ around
the points $P_i$ where we leave out one of them.
For example we can take $[C_i]$, $i=1,\ldots, N-1$. We have the
relation
\begin{equation}
\sum_{i=1}^{N-1}[C_i]=-[C_N].
\end{equation}
But there is a better choice. 
As explained above after choosing a splitting with separating
cycle $[C_S]$ we take it as one of the  basis elements 
and $N-2$ other $[C_i]s$ which are linearly independent.

For the standard splitting with $P_N=\{\infty\}$ we have
\begin{equation}
[C_S]=-[C_\infty],
\qquad
[C_i],\ i=1,\ldots, N-2.
\end{equation}
Integration around the $C_i$ can be done via calculations of
residues. Hence we always get for our geometric cocycles
(for the standard splitting)
\begin{equation}
\ga_{[C_S]}=
\sum_{i=1}^{N-1}\res_{P_i}(\gah)=
- \res_{\infty}(\gah),
\qquad
\ga_{[C_i]}=
\res_{P_i}(\gah),\ i=1,\ldots, N-2.
\end{equation}

In the following we will define geometric cocycles for all our
algebras. Results of the author 
\cite{Schlloc}, \cite{Schlaff} shows that for those cocycles 
the $\ga_{[C_S]}$ will be local, whereas the
other $\ga_{[C_i]}$ needed for a basis will not be local.
Hence only the first type will allow us to extend the almost-grading.
For simplicity we will sometimes use
$\ga_{i}$ for  $\ga_{[C_i]}$ and $\ga_{S}$ for  $\ga_{[C_S]}$.
\subsection{Function algebra $\A$}
As $\A$ is an abelian Lie algebra all anti-symmetric bilinear forms
will
define a 2-cocycle. Moreover, there will be no 
non-trivial coboundaries. Hence
$\Ho^2(\A,\C)\equiv \bigwedge^2\A$, which is an infinite-dimensional
vector space.
\begin{definition}
A cocycle $\ga$ for $\A$ is called
\newline
(a) 
 \emph{$\La$-invariant} if and only
if 
\begin{equation}
\ga(e\ldot f, g)+\ga(f,e\ldot g)=0,\quad \forall f,g\in\A,\forall
e\in\La,
\end{equation}
(b) multiplicative if and only
\begin{equation}
\ga(fg,h)+\ga(gh,f)+\ga(hf,g)=0,
\qquad f,g,h\in\A.
\end{equation}
\end{definition}
The definitions look rather ad-hoc for the moment, but we will
find the explanation of $\La$-invariance 
when we discuss the differential operator
algebra. Multiplicativity is needed for the current algebra.
Note that the fact that $\ga$ is a multiplicative cocycle for the 
commutative algebra $\A$ can also be formulated as that
it is a 1-cocycle in Connes's cyclic cohomology
$\mathrm{HC}^1(\A,\C)$.

We define the bilinear map
\begin{equation}
\A\times\A\to\Fl[1],\qquad
\gAh(f,g)=f\cdot dg,
\end{equation}
and the associated candidate for a cocycle
\begin{equation}\label{eq:exta}
\gA[C](f,g)=\cint{C}fdg.
\end{equation}
\begin{proposition}\label{P:lmhg}
(1) $\gA[C]$ is an $\La$-invariant cocycle.
\newline
(2) $\gA[C]$ is a multiplicative  cocycle.
\newline
(3) The $\gA[C_i]$ $i=1,\ldots,N-1$ define linearly independent
cocycles and hence linearly independent cohomology classes.
\end{proposition}
\begin{proof} For the statements (1) and (2) see \cite{Schlloc} and
\cite[Prop.~6.20]{Schlkn}. 
With respect to 
the standard splitting  the
cocycles  $\gA[C_i]$  are exactly those $\La$-invariant
which are bounded. In particular in the above references
it is shown, that they are linearly independent.
\end{proof}
\begin{proposition}\label{P:linvloc}
\cite{Schlloc}, \cite{Schlkn}
Given a splitting $A=I\cup O$ and the induced almost-grading, then
\newline
(1)
up to scaling 
\begin{equation}\gA[{C_S}](f,g)=\cins fdg
\end{equation}
is the unique $\La$-invariant cocycle which is local with respect to
the grading.
\newline
(2)
A basis of those $\La$-invariant cocycles which are bounded with
respect to the almost-grading
are given by the $\gA[C_i]$, $P_i\in I$.
\end{proposition}
\begin{proposition}\label{P:multloc}
\cite{Schlloc}, \cite{Schlkn}
(1) The statements of \refP{linvloc} are true
also for multiplicative cocycles.
\newline
(2) Every bounded cocycle which is $\La$-invariant
is multiplicative and vice versa.
\end{proposition}
The identifications of both types are done 
individually by reducing them to 
expressions as sums of  $\gL[C_i]$. Those have both properties. 

The propositions above are true for all genera.
Note that by calculating residues it is possible to calculate the values.
As an illustration we will do this in \refS{3point}.
\begin{theorem}\label{T:tfclass}
Let $\ga$ be an $\La$-invariant or multiplicative 
cocycle for the multi-point
function algebra in genus zero, then
\newline
(1) $\ga$ is a linear combination of geometric cocycles
of the type 
\begin{equation}
\gA[i](f,g)=\cint{C_i}fdg=\res_{a_i}(fdg),
\quad i=1,\ldots, N-1.
\end{equation}
(2) $\ga$ is bounded from above (by zero) with respect to the
almost-grading given by the standard splitting.
\newline
(3) Every $\La$-invariant cocycle is multiplicative and vice versa.
\end{theorem}
Before we start with the proof 
we quote
\begin{proposition}\label{P:mlf}
In the classical situation $g=0$, $N=2$ every $\La$-invariant cocycle
$\ga$ is multiplicative and vice versa. It is given by
\begin{equation}
\ga(f,g)=\alpha\cdot \cint{C}fdg=\alpha\cdot\res_0(fdg),\quad \alpha\in\C,
\end{equation}
and $C$ a circle around $0$.
Moreover,
\begin{equation}\label{eq:anm}
\ga(A_n,A_m)=\alpha \cdot (-n)\cdot \delta_m^{-n}.
\end{equation}
In particular $\ga$  is local and bounded by zero.
\end{proposition}
\begin{proof}
\cite[Prop.~6.50, Rem.~6.64]{Schlkn},
\cite{Schlloc}.
\end{proof}
Next we
present some
general techniques which will be used also in the proofs
of the related statements for the other algebras.
We consider the standard splitting. Recall that this says
\begin{equation}
\{P_1,P_2,\ldots, P_{N-1}\}\cup \{P_N=\infty\},
\end{equation}
where $P_i$ is the point given by the coordinate $a_i\in\C$.
The function algebra $\A$  decomposes with respect to
the basis elements
\begin{equation}
\A=\bigoplus_{n\in\Z}\A_n,\qquad
\A_n:={\langle A_{n,1},\ldots, A_{n,N-1}\rangle}_\C.
\end{equation}
We introduce the associated filtration
\begin{equation}
\A_{(n)}=\bigoplus_{m\ge n}\A_m.
\end{equation}
In \cite{Schlkn}, \cite{SchlHab} we showed that
\begin{equation}
\A_{(n)}=\{f\in \A\mid \ord_{P_i}(f)\ge n,\ i=1,\ldots, N-1\}.
\end{equation}
In particular, $\A_{(0)}$ is the subalgebra of 
meromorphic functions which are holomorphic on the affine part.
We already introduce the elements
$A_n^{(i)}:=(z-a_i)^n,\quad n\in\Z$.
Recall that beside the KN type basis elements  the vector space
$\A_{(0)}$ can also be generated by 
\begin{equation}
{\mathcal B}_i=\{A_n^{(i)}\mid n\ge 0\}
\end{equation} 
for any fixed $i=1,\ldots, N-1$.
\begin{proposition}\label{P:andec}
\begin{equation}
\begin{gathered}{}
\A_{(n)}=
\bigcap_{i=1,\ldots, N-1}{\langle 
A_k^{(i)}\mid k\ge n\rangle}_\C,\qquad \text{for}\ n\ge 0
\\
\A_{(n)}=
\sum_{i=1}^{N-1}{\langle A_k^{(i)}\mid k\ge n\rangle}_{\C},
\qquad\text{for}\ 
 n<0 \ .
\end{gathered}
\end{equation}
\end{proposition}
\begin{proof}
Let $n\ge 0$. Then the elements in the intersection 
fulfill the order prescription $\ord_{P_i}(f)\ge n$.
Vice versa every KN type basis elements lying in 
$\A_{(n)}$ can be expressed as linear combinations of
powers $(z-a_i)^k$  with $k\ge n$ (just take the Taylor 
expansion at $a_i$).
For negative $n$ again the elements from the sum on the r.h.s.
fulfill the order description, as
$\ord_{P_j}(A_k^{(i)})=0>n$ for $j\ne i$ 
and $\ord_{P_i}(A_k^{(i)})\ge n$ for $k\ge n$.
By the expansion into partial fractions the KN 
type basis elements
on the l.h.s. are linear combinations of the elements
on the r.h.s. Hence equality.
\end{proof}
We point out that the vector space sum above will be not
a direct sum (at least if $N>2$).
The corresponding statements are of course true also for
the spaces $\Fl$, as for them the basis can be identified
with the basis of $\A$ up to a $\lambda$-depending shift.

\medskip 
\begin{proof} (of \refT{tfclass}
Let $\ga$ be either $\La$-invariant or multiplicative. The statement
of \refP{mlf} will be valid for all $A_n^{(i)}$, $i=1,\ldots, N-1$.
We consider the values of 
$\ga(\A_{(m)},\A_{(m')})$ for $m+m'>0$ and will show that they will
vanish.
Necessarily either $m$ or $m'$ has to be $>0$. We assume that
$m\ge 1>0$, in particular $f\in\A_{(m)}$ will be 
a linear combination of elements $A_k^{(i)}$ with $k\ge 1$
(see \refP{andec}) for every $i$.
If $m'\ge 0$ then  $g\in\A_{(m')}$ will be a linear combination
of $A_k^{(i)}$ and    \refP{mlf} shows that indeed
\begin{equation}\label{eq:anmres}
\ga(\A_{(m)},\A_{(m')})=0.
\end{equation}
It remains to consider $m'<0$. In this case let $f\in\A_{(m)}$
as above and take $A_k^{(i)}$, $k\ge m'$ for an arbitrary $i$.
With respect to this $i$ we  write
$f=\sum_{r\ge m}\alpha_rA_r^{(i)}$. Hence,
\begin{equation}
\ga(f,A_k^{(i)})=
\ga(\sum_{r\ge m}\alpha_rA_r^{(i)},A_k^{(i)})
=\sum_{r\ge m}\alpha_r\cdot\ga(A_r^{(i)},A_k^{(i)})=0,
\end{equation}
by \refP{mlf} as $r+k\ge m+m'>0$.
This is true for all $i=1,\ldots, N-1$, hence we get \refE{anmres} too.  
In this way we showed that every $\La$-invariant or multiplicative
cocycle is bounded with respect to the almost-grading given by
the standard splitting. For bounded cocycles of this type the
author showed in \cite{Schlloc}, see also \cite{Schlkn}, that
they are geometric cocycles of the claimed form.
In particular, both types of cocycles coincide.
\end{proof}

\begin{corollary}
In the $N$-point genus 
zero situation the space of $\La$-invariant (or multiplicative)
cocycles is $(N-1)$
dimensional and is isomorphic to $\Ho_1(\Sigma_0\setminus A,\C)$ via
\begin{equation}
[C]\mapsto \gA[C];\qquad
\gA[C](f,g)=\cint{C}fdg.
\end{equation}
In particular, every $\La$ invariant cocycle or multiplicative 
cocycle is geometric.
\end{corollary}

The following is also part of the above classification.
\begin{proposition}
With respect to the standard splitting up to
rescaling the cocycle
\begin{equation}
\gA[\infty]=-\sum_{i=1}^{N-1}\gA[C_i](f,g)=\res_\infty(fdg)
\end{equation}
is the unique $\La$-invariant (and equivalently multiplicative)
cocycle which is local.
\end{proposition} 
\begin{remark}
Let $I$ be a non-empty subset of $\{1,2,\ldots, N\}$ then 
the cocycle 
$\gA[I]:\sum_{i\in I}\gA[C_i]$ can be made to a local
cocycle by taking the splitting
$I$ and  $\{1,2,\ldots, N\}\setminus I$
and choosing the induced almost-grading.
\end{remark}
\begin{remark}
\emph{(Heisenberg algebra.)}
Of course, there does not exists a universal central extension
of $\A$.
But we could consider the algebra with central terms coming from
the $\La$-invariant or equivalently
multiplicative ones.
The corresponding central extension of $\A$ will have a
$(N-1)$-dimensional center and will be given as 
\begin{equation}
[\widehat{f}, \widehat{g}]=
\sum_{i=1}^{N-1}\alpha_i\cdot\gA[C_i](f,g)\cdot t_i,
\quad\alpha_i\in\C,
\quad [t_i,\Ah]=0\, .
\end{equation}
The local cocycle $\gA[C_S]$ will yield a one-dimensional 
central extension which is almost-graded. 
One might either call the $(N-1)$-dimensional central extension
or the almost-graded one-dimensional central
extension \emph{(multi-point) Heisenberg algebra}.

\end{remark}

\begin{conjecture}
For the higher genus and multi-point situation
and $\ga$ a cocycle for $\A$ 
the following statements are equivalent:
\begin{enumerate}
\item
 $\ga$ is $\La$-invariant
\item
  $\ga$ is multiplicative
\item
$\ga$ is a geometric cocycle.
\end{enumerate}
\end{conjecture}
\refP{lmhg} shows that Condition 3 implies the other two. But
for the opposite  the  proofs presented above  in genus zero will not work as 
there are geometric cocycles which are not bounded.

\subsection{Current and affine algebras}\label{sec:aff}
First note that if $\g$ is a perfect Lie algebra,
i.e. $[\g,\g]=\g$,  then $\gb$ is perfect too. 
Hence, if $\g$ is semi-simple 
the current algebra $\gb$ admits a universal
central extension.
For current algebras introduced in \refS{currint}
associated to a finite-dimensional Lie algebra $\g$ geometric
2-cocycles can be given in the following way.
First we fix $\beta$ a symmetric, invariant, bilinear form for $\g$.
In particular $\beta([x,y],z)=\beta(x,[y,z])$.
Then
\begin{equation}
\ggb[\beta,C](x\otimes f, y\otimes g)
=\beta(x,y)\cdot \gA[C](f,g)=
\beta(x,y)\cdot \cint{C}fdg
\end{equation}
defines a geometric cocycle for $\gb$.
Here the multiplicativity of $\int_C fdg$ is crucial.

If $\g$ is simple then $\beta$ is necessarily  a multiple
of the Cartan-Killing form. In this case Kassel 
\cite{Kas82}, \cite{Kas84} proved
that the algebra $\gb=\g\otimes \A$ for any commutative algebra
$\A$ admits a universal central extensions which is given
by
\begin{equation}
\gh^{univ}=\gb\oplus (\Omega_\A^1/d\A)
\end{equation}
with Lie structure
\begin{equation}
[x\otimes f,y\otimes g]=
[x,y]\otimes fg+\beta(x,y)\overline{fdg}.
\end{equation}
Here $\Omega_\A^1/d\A$ denotes the vector space of 
K\"ahler differentials of the algebra $\A$, and $\overline{fdg}$ means
$fdg$ modulo exact differentials.
As in our situation $\Sigma_g\setminus A$ is an affine curve using 
Grothendieck's algebraic deRham theorem 
(see also Bremner \cite{Brem2} \cite{Brem3} and 
the author \cite{Schlaff}) we can dualize the space
by integrating  $\overline{fdg}$ over closed curves $C$. In this way
we identify the center with 
$\Ho_1(\Sigma_g\setminus A,\C)$.
More precisely, let $[C^{(k)}]$, $k=1,\dots, 2g+N-1$ a basis of the
this homology space represented by cycles, then
the universal central extension is given via  geometric cocycles
as 
\begin{equation}\label{eq:currunivg}
[x\otimes f,y\otimes g]=
[x,y]\otimes fg+\beta(x,y)\cdot
\sum_{i=1}^{2g+N-1}\alpha_i\cdot\cint{C^{(i)}}fdg\cdot t_i,
\quad
\alpha_i\in\C
\end{equation}
with $t_i$ central.

Now we return to the genus zero case where
 all geometric cocycles can be written with the help of residues.
\footnote{Note that this is not the case for higher genus.}
\begin{proposition}
The universal central extension has
a $(N-1)$-dimensional center and is given as
\begin{equation}\label{eq:curruniv}
[x\otimes f,y\otimes g]=
[x,y]\otimes fg+\beta(x,y)\cdot
\sum_{i=1}^{N-1}\alpha_i\cdot\res_{a_i}(fdg)\cdot t_i,
\quad
\alpha_i\in\C
\end{equation}
with $t_i$ central.
\end{proposition}
The one-dimensional central extensions 
are given (up to equivalence) as 
\begin{equation}
[x\otimes f,y\otimes g]=
[x,y]\otimes fg+\beta(x,y)\cdot
\left(\sum_{i=1}^{N-1}\alpha_i\res_{a_i}(fdg)\right)\cdot t, \quad \alpha_i\in\C.
\end{equation}
It is shown in \cite{Schlloc}, \cite[Thm.9.2]{Schlkn} that the
cocycle class $\ga$ is local if and only if it can be represented 
by
\begin{equation}\label{eq:curloc}
\alpha \cdot \beta(x,y)\cdot\sum_{P\in I}\res_P(fdg).
\end{equation}
By requiring $\La$-invariance for $\ga$
in the sense of \cite[Def.9.11]{Schlkn}:
\begin{equation}
\forall x,y\in\g,\quad
e\in\La,\ g,h\in\A:\quad 
\ga(x\otimes(e.g),y\otimes h)+
\ga(x\otimes g,y\otimes (e.h))=0
\end{equation}
we obtain even equality in \refE{curloc}
(not only up to equivalence).

Based on \refE{curruniv} it is easy to calculate central 
extensions as everything is defined in terms of the central
extensions for $\A$ (which is done by
calculating residues of rational function in the case of genus zero)
and the values of the Cartan-Killing form.

In \refA{sl2} 
we give the results in detail for the three-point situation and the 
Lie algebra $\sln(2,\C)$
on the basis of our results in \refS{3point}. See also Bremner \cite{Brem3} for 
explicite calculations in the 
4-point situation using Gegenbauer polynomials. 
\begin{remark}
The corresponding classification results are true for
$\g$ semisimple. In this case $\beta$ has to be replaced by
sums of the Cartan--Killing
forms of the components. 
As $\gb$ is also perfect it admits a universal central 
extension with a center of dimension $L\cdot(N-1)$ where $L$
is the number of simple factors.
See the above references for detailed results.
In the reductive, but not semi-simple case, $\gb$ does not
have a universal central extension, nevertheless their exists
classification results for cocycles which are $\La$-invariant
if restricted to the abelian part.
If $k$ equals the dimension of the abelian part, then the
dimension of the center is 
\begin{equation}
\big(L+\frac {k(k+1)}{2}\big)(N-1).
\end{equation}

\end{remark}

\subsection{Vector field algebra}
For the vector field algebra $\L$ we choose the following
bilinear map
\begin{equation}
\gLh[R](e,f)
=(\frac 12(ef'''-ef''')-R(ef'-e'f)dz.
\end{equation}
Here again we identified the vector fields with their local
representing functions and the element $R$ is a projective connection
(holomorphic outside of $A$), see
\refA{connect} for its definition.
Only by the term coming with $R$ it will be a well-defined
differential.
The associated geometric cocycles 
are given by
\begin{equation}
\gL[{C,R}]=\cint{C}\gLh,\qquad [C]\in\Ho_1(\Sigma_g\setminus A,\C).
\end{equation}
\begin{proposition}\cite{SchlDiss},\cite{Schlkn}
(1) The bilinear form $\gL[{C,R}]$ is a 2-cocycle for $\La$. 
\newline
(2) A different choice of the projective connection
will yield a cohomologous cocycle.
\end{proposition}
Let us quote the general result.
\begin{theorem}
If $N\ge 1$ then $\La$ admits a universal central extension
with a $(2g+N-1)$ dimensional center. More precisely,
\begin{equation}
\Ho^2(\La,\C)\cong \Ho_1(\Sigma_g\setminus A,\C)
\end{equation}
where the isomorphism is given by
$[C]\mapsto\gL[C]$.
\end{theorem}
For the proof we refer to Skryabin 
\cite{Skry}.
It is related to the Novikov conjecture. For more details see
also \cite[\S 6.10.2]{Schlkn}.

\medskip
Here we present in genus zero a direct and elementary proof of
the above theorem which has the advantage to provide
additional information.
First 
with respect to local coordinates which are
related via projective linear transformations we can choose
$R\equiv 0$ on all such coordinate patches.
This condition is fulfilled with respect to our standard covering
of $\poc$. Hence, we might ignore the additional term.

Furthermore, the calculation efforts can be reduced by
\begin{lemma}\label{L:resum}
 For $a\in\Sigma_0$
\begin{equation}
1/2\res_{a}((ef'''-e'''f')dz)=\res_{a}(ef'''dz)
=-\res_{a}(e'''fdz)\, .
\end{equation}
\end{lemma}
\begin{proof}
Using the rules for differentiation of products we
see
\begin{equation}
(ef)'''=e'''f+3e''f'+3e'f''+ef'''
=e'''f+ef'''+3(e'f)'.
\end{equation}
As derivatives of functions do not have residues 
$\res_{a}(e'''f)=-\res_{a}(ef''')$. Hence the claim.
\end{proof}
\begin{theorem}\label{T:vbound}
Every cocycle $\ga$ for the algebra $\L$ in the genus
zero,  multi-point case is cohomologous to a 
bounded one with respect to the almost-grading given by the
standard splitting.
\end{theorem}
\begin{proof}
We recall the structure of the proof of \refP{andec}. We  have to take
into account that the degree is shifted. In particular 
$\L_{(-1)}$ corresponds to those vector fields which are holomorphic
on the affine part. Note that we have the generator set
(which is not a basis)
\begin{equation}
e_{n}^{(i)}=(z-a_i)^{n+1}\ddz,\quad n\in\Z,\ i=1,\ldots, N-1.
\end{equation}
For each $i$ separately, this is 
the usual Witt algebra and we have
\begin{equation}
[e_{n}^{(i)},e_{m}^{(i)}]=(m-n)e_{n+m}^{(i)},\qquad
[e_{0}^{(i)},e_{m}^{(i)}]=m\cdot e_{m}^{(i)}\, .
\end{equation}
We start with an arbitrary cocycle $\ga$ and make cohomological
changes.
To this end we define a linear map
$\psi:\La\to\C$ and modify $\ga$ by the corresponding
coboundary $d_1\psi$.
We prescribe $\psi$ on a set of basis elements.
For this we
single out $i=1$ and take as basis of $\La$ 
\begin{equation}
e_n^{(1)},\ n\ge -1,\qquad
e_m^{(i)},\ m\le -2,\ i=1,\ldots, N-1\, .
\end{equation}
We set
\begin{equation}\label{eq:psider}
\begin{aligned}{}
\psi(e_n^{(1)})&:=\frac 1n\ga(e_0^{(1)},e_{n}^{(1)}), \quad n=-1,1,2,...
\\ 
\psi(e_0^{(1)})&:=\frac 12\ga(e_{-1}^{(1)},e_{1}^{(1)}), 
\\ 
\psi(e_n^{(i)})&:=\frac 1n\ga(e_0^{(i)},e_{n}^{(i)}), \quad n\le -2\, .
\end{aligned}
\end{equation}
The cohomologous cocycle is now
\begin{equation}
\ga'(e,f)=\ga(e,f)-\psi([e,f]).
\end{equation}
By construction we obtain
\begin{equation}\label{eq:afterlcob}
\ga'(e_0^{(1)},e_n^{(1)})=0,\quad \forall n,\qquad
\ga'(e_0^{(i)},e_n^{(i)})=0,\quad \forall n\le -2,
\ \forall i,\qquad
\ga'(e_{-1}^{(1)},e_1^{(1)})=0 \, .
\end{equation}
To avoid cumbersome notation we use $\ga$ again for the cohomologous
cocycle $\ga'$.
We write down the cocycle condition for each $i$:
\begin{equation}
\ga([e_l^{(i)},e_k^{(i)}],e_0^{(i)}])
+
\ga([e_k^{(i)},e_0^{(i)}],e_l^{(i)}])
+
\ga([e_0^{(i)},e_l^{(i)}],e_k^{(i)}])=0\,.
\end{equation}
This evaluates to
\begin{equation}\label{eq:lmother}
(k-l)\ga(e_{k+l}^{(i)},e_0^{(i)})
-
(k+l)\ga(e_{k}^{(i)},e_l^{(i)})
= 0.
\end{equation}
First we consider pairs of elements which are holomorphic in the
affine part, i.e. the space $\La_{(-1)}\times \La_{(-1)}$
and show 
\begin{equation}\label{eq:claimv}
\ga(f,g)=0,\qquad f,g\in \La_{(-1)}.
\end{equation}
By \refE{afterlcob} $\ga(e_n^{(1)},e_0^{(1)})=0$ for all $n$ and 
\refE{lmother} implies $\ga(e_k^{(1)},e_l^{(1)})=0$ for $l\ne -k$.
This shows \refE{claimv} with the exception of
 $\ga(e_1^{(1)},e_{-1}^{(1)})=0$, which is true by our cohomological
changes \refE{afterlcob}.
In the next step we show
\begin{equation}\label{eq:vfstate}
\ga(e_k^{(i)},e_l^{(i)})=0,\quad \text{for} \quad k+l>0.
\end{equation} 
As long as $k,l\ge -1$ both vector fields are holomorphic and the
claim is true by \refE{claimv}.
Now assume $l\le -2$. For $k+l>0$ to be true $k\ge 3$  is necessary. In
particular both $(k-l)$ and $(k+l)\ne 0$ and $e_{k+l}^{(i)}$ is
holomorphic.
{}From \refE{lmother} the Equation \refE{vfstate} follows.
\footnote{
Indeed, using the second relation in \refE{afterlcob} we obtain
\refE{vfstate} to be true for all $k+l\ne 0$,
This statement is not needed in the proof, but see \refR{copies}.}
For the claim of the theorem we have to show that 
if $f\in\La_{(n)}$ and   $g\in\La_{(m)}$ with $n+m>0$ then
$\ga(f,g)=0$. As $n+m>0$ at least one of them has to be $>0$.
Assume $n>0$. If $m\ge -1$ then both $f$ and $g$ are holomorphic hence
the
claim with \refE{claimv}.
We assume $m\le -2$. By \refP{andec} the elements
$e_k^{(i)}$, $k\ge m$, $i=1,\ldots, N-1$ generate
$\La_{(m)}$. Hence $g$ is a linear combination of those.
Fix one $i$ and one $e_k^{(i)}$ in this range.
By \refP{andec} the corresponding $f\in\L_{(n)}$ can be written as  
a linear combination 
$\sum_l\alpha_le_l^{(i)}$
of $e_l^{(i)}$ with $l\ge n\ge -m$.
Now 
\begin{equation}
\ga(f,e_k^{(i)})=
\ga(\sum_l\alpha_le_l^{(i)},e_k^{(i)})
=\sum_l\alpha_l\ga(e_l^{(i)},e_k^{(i)})
=0,
\end{equation}
by \refE{vfstate} as $l+k\ge n+m>0$.
As $g$ is a linear combination of such $e_k^{(i)}$ the claim of the
theorem follows.
\end{proof}
\begin{remark}\label{R:copies}
In the proof above the prescription of \refE{psider} involving
$n\le -2$ was not necessary. We could have taken also the value
0 there. Nevertheless, the prescription given has the side-effect that
the cohomological cocycle $\ga'$   
restricted to the $N-1$ individual subalgebras $\langle e_n^{(i)}\mid
n\in\Z\rangle$
which are copies of the Witt algebra, is there a multiple of the
standard Virasoro cocycle (which is centered in degree zero).
Here a warning is at order. We  showed that the cocycle $\ga'$ is
bounded from above by zero. It will not be true in general, that the
$\ga'$ is also bounded from below by zero (meaning local). 
In general $\ga'(e_n^{(i)},e_m^{(j)})\ne 0$ for $n,m\le -2$.
Moreover, by the results of \cite{Schlloc} it will be bounded from
below
too if and only if there exists an $\alpha\in\C$ such
that $\ga'=\alpha\sum_{i=1}^{N-1}\gL[i]$.
\end{remark}
\begin{theorem}
The cohomology space $\Ho^2(\La,\C)$ for the $N-$point genus 
zero vector field algebra $\La$ is  $(N-1)$-dimensional.
A basis is given by the cohomology classes
\begin{equation}
[\gL[i]]=[\gL[C_i]], \quad i=1,\ldots,N-1, 
\end{equation}
or equivalently 
\begin{equation}
[\gL[i]], \quad i=1,\ldots,N-2,
\qquad
 [\gL[\infty]]\;.
\end{equation}
In particular all cohomology classes are geometric cocycles.
The cocycle $\gL[\infty]$ will be local with respect to the
standard splitting.
\end{theorem}
\begin{proof}
By \refT{vbound} all cocycles are cohomologous to bounded
cocycles with respect to the standard splitting.
Hence the bounded cohomology classes constitute a basis of
$\Ho^2(\La,\C)$. With respect to any splitting the bounded cocycle
classes have been classified by the author in \cite{Schlloc}, see also
\cite{Schlkn}.  In particular it is shown there that the 
noted cohomology classes are a basis of the bounded cocycles
with respect to the standard splitting. In addition it is shown that
the class $[\gL[\infty]]$ is up to a scalar multiple the only cocycle
which is local with respect to the standard splitting.
\end{proof}

\begin{theorem}
For $g=0$ the algebra $\La$ is perfect, hence it admits a universal
central extension generated by the above geometric cocycles.
\end{theorem}
\begin{proof}
We know that $\La$ is generated by the $e_n^{(i)}$. As 
$[e_o^{(i)},e_n^{(i)}]=n\cdot e_n^{(i)}$ and 
$[e_{-1}^{(i)},e_1^{(i)}]=2\cdot e_0^{(i)}$ all generators
can be written as commutators, i.e. $[\L,\L]=\L$.
As a perfect Lie algebras $\La$ admits a universal central extensions
parameterized
by $\Ho^2(\La,\C)$. 
The
universal central extension is freely generated by the cohomology classes
given by the cocycles described above.
\end{proof}
We recall that these cocycles can be explicitely calculated
via residues, e.g.
\begin{equation}
\begin{aligned}{}
\gL[i](e,f)&=\res_{a_i}(ef'''dz),\quad i=1,2,\ldots, N-2,
\\
\gL[\infty](e,f)&=
-\sum_{i=1}^{N-1}\res_{a_i}(ef'''dz)
=\res_{\infty}(ef'''dz).
\end{aligned}
\end{equation}
As above with respect to the standard splitting and almost-grading 
$\gL[\infty]$ will be the unique local cocycle (up to equivalence and
rescaling). The others $\gL[i]$ are not local, but bounded from 
above.
In \refS{3point} we will do some explicit calculations.
\begin{remark}
A description of the universal central extension was also given
by Cox, Guo, Lu, and Zhao \cite{CGLZ}. Their approach is again
different. By intelligent guesses they prescribe for pairs of 
the basis elements 
from \refS{unhandy}  $(N-1)$ bilinear maps and verify directly
that these define 2-cocycles which are not coboundaries
and are linearly independent and all cocycles can be written by them.
For this quite involved combinatorial calculations are needed.

On the basis of \cite{CGLZ} Jurisich and Martens treated the
3-point case \cite{JM} in detail by calculating the cocycles 
explicitely. Quite involved binomial identities are needed. 
We show in \refS{3point} how by consequently using the multi-point
KN type algebra the cocycles can be calculated in a much simpler way.
\end{remark}

\subsection{Differential operator algebra}
For the differential operator algebra $\Do$ the cocycles 
of type \refE{exta} for $\A$ can be
extended to $\Do$ by zero on the subspace $\La$.
For this we need that they are $\La$-invariant. 
The cocycles for $\La$ can be pulled back.
In addition there is a third type of cocycles mixing $\A$ and $\La$:
\begin{equation}\label{eq:mix}
\gm[{C,T}](e,g):=
\cint{C}(e g''+T eg')dz,
\qquad  e\in\La,g\in\A,
\end{equation} 
with an affine connection $T$, with at most a pole of order one
at a fixed point in $O$ (see \refA{connect} for the definition
of an affine connection)
\begin{proposition}\cite{Schlloc}, \cite{Schlkn}
$\gm[{C,T}]$ defines a 2-cocycle for 
$\Do$. Another choice of a connection $T$ will not change its
cohomology class.
\end{proposition}
In the above references we showed that every cocycle of $\Do$ can be 
uniquely decomposed into a function algebra cocycle, a vector field
cocycle and cocycle of mixing type.
Also a complete description of bounded and local mixing cocycles is
given.
It follows exactly the same pattern as in the vector field case.
\begin{remark}
It is possible to choose $T=0$ if all coordinate transformations
are affine. But this is only true in genus 1.
For different genera our $T$ will be meromorphic. But it is
possible to do with a pole of order one, which in our
situation we put at the point $\infty$.
In fact here we can do with 
$(T(z)=0, T(w)=-2/w)$.
Hence, in the affine part we do not see the appearance of $T$.
For calculations in our genus zero with 
our system of coordinates we can completely ignore
it (this is the same with a projective connection).
After having verified that $\gmh(e,f)=(eg''+Teg')dz$ is a well-defined differential
it is uniquely given by the rational function representing it
on the affine part. There the connection is equal to 0.
The behaviour of the rational function at $\infty$ is given by the
transformation
law of 1-differentials, and 
we can calculate e.g. the residue there
in the normal way.
\end{remark}
For the genus zero situation again
$[\gm[i]]=[ \gm[{C_i}]]$ for $i=1,\ldots N-1$ will be a basis 
of the bounded mixing cocycles with respect to the
standard splitting and $\gm[\infty]=\gm[{C_S}]$ the essentially
unique local cocycle (up to rescaling and equivalence).
In more detail
\begin{equation}
\begin{aligned}{}
\gm[i](e,g)&=\res_{a_i}(eg''dz),\quad i=1,2,\ldots, N-2,
\\
\gm[\infty](e,g)&=
-\sum_{i=1}^{N-1}\res_{a_i}(eg''dz)
=\res_{\infty}(eg''dz).
\end{aligned}
\end{equation}
\begin{theorem}\label{T:mbound}
Every mixed cocycle for the differential operator algebra in the
genus zero and multi-point situation is cohomologous to
a bounded cocycle with respect to
the standard splitting.
Furthermore, this bounded cocycle is a linear combination of the 
geometric cocycles $\gm[i]$, $i=1,\ldots, N-1$.
\end{theorem}
\begin{proof}
We will use the same strategy as above in the vector field
algebra case.
For the elements 
\begin{equation}
e_n^{(i)}=(z-a_i)^{n+1}\ddz,\qquad
A_n^{(i)}=(z-a_i)^n
\end{equation}
we have the relations
\begin{equation}
[e_n^{(i)},A_m^{(i)}]=e_n^{(i)}\ldot A_m^{(i)}=m\, A_{n+m}^{(i)}.
\end{equation}
In particular $[e_0^{(i)},A_m^{(i)}]=m\, A_{m}^{(i)}$.
We start with an arbitrary $\ga$ and modify it by a coboundary
coming
from a linear map $\psi:A\to\C$. We set
\begin{equation}
\begin{aligned}{}
\psi(A_n^{(1)})&=\frac 1n\ga(e_0^{(1)},A_n^{(1)}),\quad n\ge 0,
\\ 
\psi(A_0^{(1)})&=\ga(e_{-1}^{(1)},A_1^{(1)}),
\\
\psi(A_n^{(i)})&=\frac 1n\ga(e_0^{(i)},A_n^{(i)}),\quad n\le -1, 
i=1,\ldots, N-1.
\end{aligned}
\end{equation}
By construction we have for the cohomologous cocycle 
$\ga'=\ga-d_1\psi$
\begin{equation}
\begin{gathered}
\ga'(e_0^{(1)},A_n^{(1)})=0, \ \forall n\ne 0, 
\quad
\ga'(e_0^{(i)},A_n^{(i)})=0, \ \forall n\le -1,\ i=2,\dots, N-1,
\quad
\\
 \ga'(e_{-1}^{(1)},A_1^{(1)})=0\;.
\end{gathered}
\end{equation}
We use  again $\ga$ instead of $\ga'$.
Using the cocycle relation
\begin{equation}
\ga([e_0^{(i)},e_m^{(i)}],A_n^{(i)})+
\ga([e_m^{(i)},A_n^{(i)}],e_0^{(i)})+
\ga([A_n^{(i)},e_0^{(i)}],e_m^{(i)})=0
\end{equation}
we obtain
\begin{equation}\label{eq:mmother}
(m+n)\cdot \ga(e_m^{(i)},A_n^{(i)})=
n\cdot\ga(e_0^{(i)},A_{n+m}^{(i)}).
\end{equation}
First we show that $\ga$ will vanish on holomorphic pairs, e.g. on
$\La_{(-1)}\times A_{(0)}$. 
{}From \refE{mmother} it follows that if $k+l\ne 0$ then
$\ga(e_k^{(i)},A_l^{(i)})=0$. By the cohomological changes 
also  
$\ga(e_{-1}^{(1)},A_1^{(1)})=0$.
{}From these properties we can conclude with nearly the same words
as in the proof of \refT{vbound} that
\begin{equation}
\ga(f,g)=0,\quad f\in\La_{(k)},\ g\in\A_{(l)}, \ k+l>0.
\end{equation}
We do not reproduce these calculations here.

To conclude the proof of the theorem we use the
classification results of the author on bounded cocycles
from \cite{Schlloc} with respect to the standard splitting
and obtain the claimed representation.
\end{proof}

\begin{proposition}
The differential operator algebra $\Do$ in the genus zero
and multi-point case is perfect.
\end{proposition}
\begin{proof}
As in the vector field algebra case all generators are themselves
commutators. This we showed above for the $e_n^{(i)}$. For the
others we obtain
\begin{equation}
A_m^{(i)}=\frac 1m[e_0^{(i)}, A_m^{(i)}],\ m\ne 0,\quad
A_0^{(i)}=[e_{-1}^{(i)}, A_1^{(i)}]\; .
\end{equation}
\end{proof}

\begin{theorem}
The differential operator algebra $\Do$ in the genus zero and
multi-point
situation admits a universal central extension. It has a $3(N-1)$
dimensional
center. The defining cocycles are given by the linearly independent
cocycle classes  of
\begin{equation}
\gA[i],\quad,\gL[i],\quad \gm[i],\quad i=1,\ldots, N-1,
\end{equation}
or equivalently
\begin{equation}
\gA[i],\quad,\gL[i],\quad \gm[i],\quad i=1,\ldots, N-2,\infty.
\end{equation}
In the latter presentation the 
three cocycles 
$\gA[\infty],\gL[\infty],\gm[\infty]$
are the unique basis cocycles generating the 
3-dimensional cohomology space yielding local classes
with respect to the standard splitting.
\end{theorem}

\begin{remark}
Also for higher genus we have 
for every element from $\Ho_1(\Sigma_g\setminus A,\C)$ 
three different type of cocycles for the differential operator algebra.
With the help of them we can construct a central extension 
which has a $3\cdot \dim \Ho_1(\Sigma_g\setminus A,\C)$ dimensional 
center.
In accordance to the vector field case in the higher genus
and the results in genus zero which we have just shown, 
it is reasonable to  conjecture
that this central extension is a universal central extension of $\Do$.
\end{remark}

\subsection{Central extensions of $\g$-differential operator algebras}

Assume that $\g$ is a simple Lie algebra. Then every cocycle
$\ga$ of $\gb$ 
can be made  $\La$-invariant by cohomological changes.
Such cocycle can be extended by zero on the complementary spaces 
involving elements form $\La$. Every cocycle of $\La$ defines
a cocycle of $\Dg$ by pulling it back.
In fact, every cocycle of $\Dg$ is cohomologous to a sum
of these two types of cocycles. 
As both $g$ and $\La$ are perfect, $\Dg$ is perfect too. 
Hence, it admits a universal central extension where the
center is given by the two types of cocycles.
{}From our results, presented here, respectively
\cite{Schlaff} it follows that this universal
central extension has a center of dimension $2(N-1)$.
The explicite cocycles were given above.

From the analysis in \cite{Schlaff}, \cite{Schlkn} it follows that
in the case that $\g$ is semi-simple with $L$ simple factors
the universal central extension has a center of
dimension $(L+1)(N-1)$.
In the reductive but not semi-simple case the situation is 
a little bit more involved, as the general classification
shows that mixing cocycles show up which come with a linear 
form on $\g$ vanishing on $[\g,\g]$.
Now  we have to put $\La$-invariance in the conditions 
to obtain a complete classification. 
If $k$ is the dimension of the abelian factor, then
the center will have dimension
\begin{equation}
(L+\frac {k(k+1)}{2}+1)(N-1).
\end{equation}
Of course now it is not
a universal extension anymore.
See \refR{gln} for the example of $\gl(n)$.

\section{The three-point and genus zero case}
\label{sec:3point}
%
%
\subsection{Symmetries}
The case of only three points where poles are allowed is
to a certain extend special as we have additional symmetries.
These symmetries can be used to simplify the calculations 
of structure constants even
further.

Additionally, the three-point cases play a role in quite a number
of applications. See e.g. the tetrahedron algebra appearing in
statistical mechanics, in particular the work of Terwilliger
and collaborators \cite{HaTe}, \cite{BeTer}.

By a complex automorphism of the Riemann sphere,
i.e.
by a fractional linear transformation, respectively by an
$\PGL(2)$ action the three points can be brought to the points
$0,1$ and $\infty$. The corresponding automorphism will yield
an isomorphism of the involved algebras. 
Even after this is done there are still automorphisms of $\poc$ 
permuting the 3 points $\{0,1,\infty\}$. Hence, we still
have the action of the symmetric group $S_3$ of 3 elements.
The corresponding algebraic maps are now automorphisms of the
algebras.

Recall that in the previous section we constructed for every splitting
of
the set $A$, here $\{0,1,\infty\}$, into two disjoint non-empty subset
$I$ and $O$ a distinguished basis which yields an almost-graded
structure for the algebras. 
Essentially different splittings will yield essentially different 
basis elements respectively essentially different almost-gradings.

Here the only type of splitting is into a subset 
consisting of two points and
a subset consisting of one point. 
After having fixed $A=\{0,1,\infty\}$ we can by
applying an automorphism from the remaining group $S_3$ such that
\begin{equation}
A=I\cup O,\qquad 
I:=\{0,1\},\quad\text{and}\quad
O:=\{\infty\}.
\end{equation}
This is exactly the situation which we will consider here.
\begin{remark}
I like to stress the fact, that this does not say, that there is
only one possible choice of an almost-grading. In fact, given the 
set $A$ and hence a unique algebra, we have 3 essentially different
splitting, hence also 3 essentially different set of basis elements
and
consequently 3 almost-gradings. 
But in the three-point situation there will be always an automorphism
of
our algebra mapping the different almost-gradings to each other.
\end{remark}
\begin{remark} In \cite{SchlDeg} the situation
$A=\{\alpha,-\alpha,\infty\}$ was considered for $\alpha\in\C$,
$\alpha\ne 0$. This changes nothing. The corresponding algebras are
isomorphic to the algebras considered here. Even our basis elements
can be identified (up to some rescaling and re-indexing) and the structure
equations remain the same (again up to scaling factors).
The reason for the choice there was that we wanted to introduce
a free parameter $\alpha$ which can be used to study degeneration
process, respectively deformation families. For $\alpha\to 0$ we
``degenerate'' to the two-point situation.
In this respect see also our joint work with 
Fialowski 
\cite{FiaSchl1}, \cite{FiaSchlaff}, \cite{FiaSchlaffo}.
Clearly, everything that will be done in this section could be
formulated also for $\{\alpha,-\alpha,\infty\}$, $\alpha\ne 0$.
\end{remark}
\subsection{The associative algebra}
The basis elements of degree $n$ of the algebra $\A$ with 
respect to our normalized splitting are (see \refS{npoint}, \cite{Schlce})
\begin{equation}
A_{n,1}(z)=z^n(z-1)^{n+1},
\qquad
A_{n,2}(z)=z^{n+1}(z-1)^{n},\quad n\in\Z,
\end{equation}
where we ignore the scaling factors.
We ``symmetrize'' and ``anti-symmetrize'' them for
each degree separately by taking
\begin{equation}\label{eq:baschang}
\begin{aligned}
A_n(z)&=A_{n,2}(z)-A_{n,1}(z)
&&=z^n(z-1)^n,
\\
B_n(z)&=A_{n,2}(z)+A_{n,1}(z)
&&=z^n(z-1)^n(2z-1),
\end{aligned}
\end{equation}
\begin{proposition}\label{P:aprod}
The associative algebra $\A$ of meromorphic functions on
the Riemann sphere $\poc$ holomorphic outside of 
$0$, $1$, and $\infty$ has as vector space basis
\begin{equation}
\{A_n,B_n\mid n\in\Z\},
\end{equation} 
with the structure equations
\begin{equation}\label{eq:aprod}
\begin{aligned}
{}
A_n\cdot A_m&=A_{n+m},
\\
A_n\cdot B_m&=B_{n+m},
\\
B_n\cdot B_m&=A_{n+m}+4A_{n+m+1}.
\end{aligned}
\end{equation}
\end{proposition}
\begin{proof}
That the elements $A_{n,1},A_{n,2}$ with $n\in\Z$ are a basis
of $\A$ is a general fact by its very construction as
Krichever--Novikov
multi-point basis in \cite{Schleg} corresponding to 
our splitting of $A$. The transformation 
\refE{baschang} is obviously a base change which
even respects the homogeneous subspaces.
By direct calculations which we moved to \refA{calc} 
we take from \refL{prod} the structure equation.
\end{proof}
Our splitting introduces a 
(strong) almost-grading
for the algebra $\A$
\begin{equation}
\A=\bigoplus_{n\in\Z}\A_n,\qquad
\A_n={\langle A_n,B_n\rangle}_{\C},\quad\dim\A_n=2.
\end{equation}
This is clear from the general construction.
But it can be easily illustrated by \refE{aprod} as
\begin{equation}
\A_n\cdot \A_m\quad\subseteq\quad \A_{n+m}\oplus\A_{n+m+1}.
\end{equation}

\bigskip
Next we want to study central extensions of $\A$  
(considered as Lie algebras)
which  
are given by geometric cocycles as introduced in
\refS{central}. We showed that 
\begin{equation}
\gA[0](f,g)=\res_0(fdg),
\qquad
\gA[\infty](f,g)=\res_\infty(fdg),
\end{equation}
constitute a basis of the geometric cocycles.
Recall that the set of geometric cocycles coincide with the $\La$-invariant
respectively multiplicative cocycles.
Note also 
that 
\begin{equation}
\gA[1]=\res_1(fdg)=-\gA[0](f,g)-\gA[\infty](f,g)
\end{equation}
and that
$\res_{a}(fdg)=-
\res_{a}(gdf)
$, $a\in\Sigma_0$.
To calculate the cocycles it is enough to do the calculation for
all type of pairs of basis elements $A_n$ and $B_m$.
The method is very simple. 
We use \refL{a1prim} to express the integrand (e.g. $A_nA_m'dz$)
in terms of $A_k$ and $B_k$ and then \refL{resinf}, respectively
\refL{amres} and \refL{bmres} to calculate its value.

We start with
\begin{proposition}\label{P:acocinf}
\begin{equation*}
\begin{aligned}
\gA[\infty](A_n,A_m)&=2n\,\delta_m^{-n},
\\
\gA[\infty](A_n,B_m)&=0,
\\
\gA[\infty](B_n,B_m)&=
2n\delta_m^{-n}+
4(2n+1)\,\delta_m^{-n-1}\;.
\end{aligned}
\end{equation*}
\end{proposition}
\begin{proof}
First we remark that in $A_nB_m'$ by 
\refL{a1prim} only multiples of $A_k$ show up and by \refL{resinf}
their residues at $\infty$ will vanish. Hence the second relation.
For the first relation we use
$A_nA_m'dz=mB_{n+m-1}dz$. By \refL{resinf}  there is only a 
non-vanishing value for the residue if $n+m-1=-1$ 
(equivalently $m=-n$)
 and it will be 
$-2m$. Hence the first relation.
For  $B_nB_m'dz$ we have 2 terms $2(2m+1)B_{m+n}$ and
$mB_{n+m-1}$. The first term will  have a residue if and 
only if $m=-n-1$ and the second if and only if $m=-n$.
This yields exactly the result of the 3. relation.
\end{proof}  
\begin{proposition}\label{P:aco0}
\begin{equation*}
\begin{aligned}
\gA[0](A_n,A_m)&=-n\,\delta_m^{-n},
\\
\gA[0](A_n,B_m)&=
n\,\delta_m^{-n}
+2n\,\delta_m^{-n-1}
+\sum_{k=2}^\infty n\,(-1)^{k-1}
2^k\frac{(2k-3)!!}{k!}\delta_{m}^{-n-k}\,
,
\\
\gA[0](B_n,B_m)&=
-n\delta_m^{-n}
-2(2n+1)\,\delta_m^{-n-1}\;.
\end{aligned}
\end{equation*}
\end{proposition}
\begin{proof}
Relation 1 and 3 are automatic as the expressions
$fdg$ are the same 
as in the proof of \refP{acocinf} and the residues 
of $B_k$ at $0$ and $\infty$ are
related via \refL{resinf}.
Hence only Relation  2 needs a calculation.
The calculation can be made a little bit simpler by using the
fact that $\res_{a}(A_ndB_m)= -\res_{a}(B_mdA_n)$. 
Using \refL{a1prim} we get
\begin{equation}
\res_{0}(A_nB_m'dz)
=-n (4\res_{0}(A_{n+m}dz)
+\res_{0}(A_{n+m-1}dz)).
\end{equation}
We set $k=-(n+m)$. From \refL{amres} we conclude that 
there is no residue for $k<0$. Next we consider 
$k\ge 2$ then 
\begin{multline}\label{eq:aress}
4\res_{0}(A_{n+m}dz)
+\res_{0}(A_{n+m-1}dz)
=
(-1)^k\left(4\frac{(2k-3)!!}{(k-1)!}2^{k-1}
-\frac{(2k-1)!!}{k!}2^{k}\right)
\\
=
(-1)^k2^k\frac{(2k-3)!!}{k!}.\qquad\qquad
\qquad\qquad
\end{multline}
This gives the result in the generic situation.
For 
the exceptional values for $k$ 
we calculate the residues in \refE{aress} separately and use
the values for the residue
\footnote{For further reference we give few more values.}
\begin{equation}
\label{eq:ressmall}
\res_0(A_0)=0,\quad
\res_0(A_{-1})=-1,
\quad
\res_0(A_{-2})=2,
\quad
\res_0(A_{-3})=-6,
\end{equation}
to obtain for 
$k=0$  the value $-1$
and for 
$k=1$ the value $-2$. 
If we multiply these values with $-n$ we get the claimed values.
\end{proof}
Note that 
the second relation in the above proposition is
expressed as a formal infinite sum. But 
for given $n$ and $m$ 
maximally one term will be non-zero.
Hence, it has a well-defined value.

In accordance with the general results 
\cite{Schlloc} 
about local cocycles only the cocycle $\gA[\infty]$ is
local. Here it will vanish for 
$n, m$  outside of $-1\le n+m\le 0$.
Consequently, only the central extension defined via 
$\gA[\infty]$ will admit an extension of the
almost-grading to the central extension.
Consequently, we obtain only in this case a triangular
decomposition which is of importance for the 
representations appearing in field theory.
This is not possible for the central extension defined
by $\gA[0]$.
\begin{remark}
If we change the almost-grading to the almost-grading introduced
by the splitting  $\{1,\infty\}\cup \{0\}$ the cocycle 
 $\gA[0]$ will become local. But we have to keep in mind that
the distinguished basis will change automatically too.
\end{remark}
\subsection{Current and affine algebra}
Recall that 
for  a finite-dimensional Lie algebra $\g$ the current
algebra of KN-type is defined by
$\gb=\g\otimes\A$. In particular every choice of a basis in 
$\g$ and a basis in $\A$ will yield a basis of $\gb$. 
For $\A$ we choose the basis $A_n,B_n$, $n\in\Z$ introduced above.
Automatically  we get an almost-graded structure of $\gb$
induced by the splitting $\{0,1\}\cup\{\infty\}$. For its
algebraic structure we obtain (via \refP{aprod})
\begin{proposition}\label{P:currprod}
\begin{equation*}
\begin{aligned}{}
[x\otimes A_n,y\otimes  A_m]&
=[x,y]\otimes A_{n+m},
\\
[x\otimes A_n,y\otimes  B_m&=[x,y]\otimes B_{n+m},
\\
[x\otimes B_n\cdot ,y\otimes B_m&=[x,y]\otimes (A_{n+m}+4A_{n+m+1}).
\end{aligned}
\end{equation*}
\end{proposition}
In \refS{npoint} 
we discussed its central extensions 
given by geometric cocycles
\begin{equation}
\ggb[\beta](x\otimes f, y\otimes g)=\beta(x,y)\cdot 
\gA(f,g),
\end{equation}
with $\beta(.,.)$ a symmetric, invariant bilinear form for
$\g$ and $\gA$ a multiplicative 2-cocycle for the algebra $\A$. 
Recall that if $\g$ is
a simple Lie there exists a universal central extension
$\gh$. 
In our case it has a  two-dimensional center and  
will be
given
by
\begin{equation}\label{eq:gext}
\begin{aligned}{}
[x\otimes f,y\otimes g]
=[x,y]\otimes (f\cdot g)
&+\alpha_0\cdot\beta(x,y) \cdot\gA[0](f,g)\cdot t_0
\\
&+\alpha_\infty\cdot\beta(x,y)\cdot\gA[\infty](f,g)\cdot t_\infty
\end{aligned}
\end{equation}
with $\alpha_0,\alpha_\infty\in\C$, $t_0,t_\infty$ central
in $\gh$, and $\beta$  the Cartan-Killing form.
The values of the cocycles for the introduced basis
elements have been calculated above and will not be repeated here.
In accordance with the general results 
\cite{Schlaff} the centrally extended
current algebra will be an almost-graded extension of the current
algebra with respect to this basis if and only if $\alpha_0=0$.
It is an easy task to write everything explicitely for 
special cases of the Lie algebra $\g$.
We will give the result for $\sln(2,\C)$ in \refA{sl2}.

As explained in \refS{central} this classification results can be
extended 
to arbitrary semi-simple Lie algebras (finite-dimensional).
If we require ``$\La$-invariance'' for the cocycle even the
reductive case is possible. 
\subsection{Vector field algebra}
Recall that in the genus $g=0$ case and $P_N=\infty$ the elements $g$
for $\Fl$ for $\l\in \frac12\Z$  are given by
\begin{equation}
g(z)=a(z)dz^{\l},\quad \text{with}\quad a(z)\in\A.
\end{equation}
In particular, a basis of $\A$ induces a basis of $\Fl$. As explained
in
\refS{npoint} a $\l$-depending shift is convenient.
We take as basis elements the elements
\begin{equation}
g_n^\l:=A_{n-\l}dz^\l,\qquad 
h_n^\l:=B_{n-\l}dz^\l,\quad 
n\in\Jl.
\end{equation}
The corresponding almost-grading reads as
\begin{equation}
\Fl=\bigoplus_{n\in\Jl}\Fln,\qquad
\Fln={\langle g_n^\l,h_n^\l\rangle}_\C.
\end{equation}
For the vector field (i.e. forms of weight $-1$)
we use
\begin{equation}
e_n:=A_{n+1}\ddz,\qquad 
f_n:=B_{n+1}\ddz,\quad 
n\in\Z.
\end{equation}
The vector spaces $\Fl$ are modules over $\La$. 
If $e=a\ddz$ and $g=b dz^\l$ then the module structure reads as 
\begin{equation}\label{eq:lie3}
e\ldot g=(a\cdot b'+\l\, b\cdot a')dz^\l.
\end{equation}
With respect to our basis elements the structure equations
are given by
\begin{proposition}\label{P:lmod}
\begin{equation*}
\begin{aligned}{}
e_n\ldot g_m^\l &=(m+\l n)\,h_{m+n}^\l,
\\
e_n\ldot h_m^\l &=(m+\l n)\,g_{m+n}^\l
+(4(m+\l n)+2)\,g_{n+m+1}^\l,
\\
f_n\ldot g_m^\l &=(m+\l n)\,g_{m+n}^\l
+(4(m+\l n)+2\l)\,g_{n+m+1}^\l,
\\
f_n\ldot h_m^\l &=(m+\l n)\,h_{m+n}^\l
+(4(m+\l n)+2+2\l)\,h_{n+m+1}^\l.
\end{aligned}
\end{equation*}
\end{proposition}
\begin{proof}
We calculate the expression \refE{lie3} for the
pairs of basis elements.
Then we use the
expressions for the derivatives from \refL{deri} and 
for the products from \refL{prod}
and obtain in a straight forward manner the results.
\end{proof}
For $\l=-1$ we get the vector field algebra structure.
\begin{proposition}\label{P:vprod}
\begin{equation*}
\begin{aligned}{}
[e_n, e_m] &=(m-n)\,f_{m+n},
\\
[e_n, f_m] &=(m- n)\,e_{m+n}
+(4(m-n)+2)\,e_{n+m+1},
\\
[f_n, f_m] &=(m- n)\,f_{m+n}
+4(m-n)\,f_{n+m+1}.
\end{aligned}
\end{equation*}
\end{proposition}
These expressions clearly exhibit the almost-graded structure.
Observe that the algebra of global holomorphic vector fields
is the subalgebra
\begin{equation}
{\langle e_{-1},f_{-1},e_0\rangle}_\C,
\end{equation}
which is isomorphic to $\sln(2,\C)$.

\bigskip

Next we calculate the universal central extension. We know
that it has a two-dimensional center and can be given as 
\begin{equation}
[\widehat{e},\widehat{f}]=
\widehat{[e,f]}+
\alpha_0\cdot \gL[0](e,f)\cdot t_o+
\alpha_\infty\cdot \gL[\infty](e,f)\cdot t_\infty
\end{equation}
with
\begin{equation}\label{eq:resv3}
\begin{aligned}{}
\gL[0](e,f)&=1/2\res_0(e\cdot f'''-f\cdot e''')dz
=\res_0(e\cdot f''')dz
=-\res_0(f\cdot e''')dz
\\
\gL[\infty](e,f)&=1/2\res_\infty(e\cdot f'''-f\cdot e''')
=\res_\infty(e\cdot f''')dz
=-\res_\infty(f\cdot e''')dz\, .
\end{aligned}
\end{equation}
Here we used \refL{resum}

First we consider the point $\infty$ and obtain in this way the
local cocycle.
\begin{proposition}\label{P:lcocinf}
\begin{equation*}
\begin{aligned}{}
\gL[\infty](e_n,e_m)&=2(n^3-n)\,\delta_m^{-n}+
4n(n+1)(2n+1)\delta_m^{-n-1}
\\
\gL[\infty](e_n,f_m)&=0,
\\
\gL[\infty](f_n,f_m)&=
2(n^3-n)\,\delta_m^{-n}+8n(n+1)(2n+1)\delta_m^{-n-1}
+8(n+1)(2n+1)(2n+3)\delta_m^{-n-2}
\end{aligned}
\end{equation*}
\end{proposition}
\begin{proof}
The method is the same as in the proof of 
\refP{acocinf}.
We have the expressions of $e_n$, respectively $f_n$ via the 
basis functions $A_n$ and $B_n$, use \refL{a3prim} and 
build the integrand in one of the forms given in 
\refE{resv3}. Finally 
\refL{resinf} will give the residue.
In the expression of $A_nB_m'''$ and $B_mA_n'''$ only 
linear combinations of $A_k$ show up. But they do not
have a residue at $\infty$. Hence, the second relation. 

The others we have to calculate.
First using \refL{a3prim}
\begin{equation}\begin{aligned}{}
\gL[\infty](e_n,e_m)&=-\res_\infty(A_{m+1}A_{n+1}''')
\\
&=-(n+1)n(n-1)\res_\infty B_{n+m-1}
-2(n+1)n(2n+1)\res_\infty B_{n+m}.
\end{aligned}
\end{equation}
Taking the residues at $\infty$ via \refL{resinf}
we get exactly the first relation.
For $\gL[\infty](f_n,f_m)=-\res_\infty(B_{m+1}B_{n+1}''')$
we obtain in completely in the same manner the last relation
(now with three terms).
\end{proof}

\begin{proposition}\label{P:lcoc0}
\begin{equation*}
\begin{aligned}{}
\gL[0](e_n,e_m)&=-(n^3-n)\,\delta_n^{-m}
-2n(n+1)(2n+1)\delta_m^{-n-1}
\\
\gL[0](e_n,f_m)&=
(n^3-n)\,\delta_m^{-n}
+6n^2(n+1)\delta_m^{-n-1}
+6n(n+1)^2\delta_m^{-n-2}
\\
&+\sum_{k\ge 3}
n(n+1)(n+k-1)(-1)^{k}2^k\cdot 3\cdot \frac {(2k-5)!!}{k!}
\delta_m^{-n-k}
\\
\gL[0](f_n,f_m)&=
-(n^3-n)\,\delta_m^{-n}-4n(n+1)(2n+1)\delta_m^{-n-1}
-4(n+1)(2n+1)(2n+3)\delta_m^{-n-2}\,.
\end{aligned}
\end{equation*}
\end{proposition}
\begin{proof}
In all cases the integrand is the same as the
one considered in \refP{lcocinf}.
Hence by, \refL{resinf} we obtain $-1/2$ of the values
there for the 1. and 3. relation.
It remains to calculate the 2nd relation.
We have
\begin{equation}
\gL[0](e_n,f_m)=-\res_0(B_{m+1}A_{n+1}''').
\end{equation}
{}From \refL{a3prim} we obtain for it
\begin{equation}\label{eq:vsumm}
-(n+1)n\,\bigg(8(2n+1)\res_0A_{n+m+1}
+2(4n-1)\res_0A_{n+m}
+(n-1)\res_0A_{n+m-1}\bigg)
\end{equation}
Next we have to consult \refL{amres} for calculating the
residues.
We set $k=-(n+m)$. If $k>0$ there will be no residue.
For $k\ge 3$ we are in the generic situation and obtain
for the last factor in \refE{vsumm}
\begin{multline}
(-1)^{-k-1}\frac{(2k-5)!!}{k!}2^k
\bigg(2(2n+1)k(k-1)
-(4n-1)(2k-3)k+(n-1)(2k-1)(2k-3)\bigg)
\\
= (-1)^{-k-1}\frac{(2k-5)!!}{k!}2^k\cdot 3\cdot (n+k-1).
\qquad\qquad
\end{multline}
This yields exactly the generic expression.
For $k=0,1$ and $2$ we have to take the 
residues in \refE{vsumm} individually and use the
the values \refE{ressmall}
and get exactly the claimed expressions.
\end{proof} 
See \cite{JM} for similar results.

Note also that we can express the term
\begin{equation}
\frac{(2k-5)!!}{k!}2^k\cdot 3
=\frac {12}{k(k-1)}\binom {2(k-2)}{k-2}
\end{equation}
if this is more convenient.

Again only the cocycle $\gL[\infty]$ will be local with respect to
the almost-grading introduced by our splitting, respectively 
by our basis.
Hence, only for its corresponding central extension we have
a triangular decomposition.
Everything what was said for the function algebra case, remains
true here.

\subsection{Differential operator algebra $\Do$}
Recall that $\Do$ is 
the (Lie algebra) semi-direct sum of $\A$ with $\La$ where $\La$ operates on $\A$ by
taking
the derivative.
The homogeneous subspace of degree $n$ is now
\begin{equation}
{\langle A_n,B_n,e_n,f_n\rangle}_\C.
\end{equation}
The subalgebra $\A$ is abelian and \refP{vprod} 
gives the structure equations for the
vector fields.  \refP{lmod} specialized for $\l=0$ yields the  
the equations for the mixed terms.
\begin{proposition}
\begin{equation*}
\begin{aligned}{}
[e_n, A_m]&=
-[A_m, e_n]
= m\,B_{m+n},
\\
[e_n, B_m]&=
-[B_m, e_n]
 = m\,A_{m+n}
+(4m+2)\,A_{n+m+1},
\\
[f_n, A_m]&=
-[A_m, f_n]
 =m\,A_{m+n}
+4m\,A_{n+m+1},
\\
[f_n, B_m]&=
-[B_m, f_n]
 =m\,B_{m+n}
+(4m+2)\,B_{m+n+1},
\end{aligned}
\end{equation*}
\end{proposition}

The geometric cocycles yield a 6-dimensional central extension. 
In
addition
to the 4 basis elements given by the pure cocycles  given by 
the Propositions \ref{P:acocinf},  
\ref{P:aco0}, \ref{P:lcocinf}, \ref{P:lcoc0} we have two additional
basis cocycles
($e\in\L,g\in\A$)
\begin{equation}
\gm[0](e,g)=\res_0(eg''dz),
\qquad
\gm[\infty](e,g)=\res_\infty(eg''dz).
\end{equation}
Evaluated for the basis elements we obtain
\begin{proposition}
\label{P:mcocinf}
\begin{equation*}
\begin{aligned}{}
\gm[\infty](e_n,A_m)&=0,
\\
\gm[\infty](e_n,B_m)&
=-2n(n+1)\delta_{m}^{-n}
-4(n+1)(2n+1)\delta_{m}^{-n-1},
\\
\gm[\infty](f_n,A_m)&=
-2n(n+1)\delta_{m}^{-n}
-4(n+1)(2n+3)\delta_{m}^{-n-1},
\\
\gm[\infty](f_n,B_m)&=0\, .
\end{aligned}
\end{equation*}
\end{proposition}
\begin{proposition}
\label{P:mcoc0}
\begin{equation*}
\begin{aligned}{}
\gm[0](e_n,A_m)&=
-n(n+1)\delta_{m}^{-n}
-2(n+1)^2\delta_{m}^{-n-1}
\\
&+
\sum_{k\ge 2} (n+1)(n+k)(-1)^{k}2^k\cdot
\frac{(2k-3)!!}{k!}\delta_{m}^{-n-k}\,.
\\
\gm[0](e_n,B_m)&=n(n+1)\delta_{m}^{-n}
+2(n+1)(2n+1)\delta_{m}^{-n-1},
\\
\gm[0](f_n,A_m)&=
n(n+1)\delta_{m}^{-n}
+2(n+1)(2n+3)\delta_{m}^{-n-1},
\\
\gm[0](f_n,B_m)&=
-n(n+1)\delta_{m}^{-n}
-6(n+1)^2\delta_{m}^{-n-1}
-6(n+1)(n+2)\delta_{m}^{-n-2}
\\
&+
\sum_{k\ge 3} (n+1)(n+k)(-1)^{k-1}2^k\cdot 3\cdot
\frac{(2k-5)!!}{k!}\delta_{m}^{-n-k}\,.
\end{aligned}
\end{equation*}
\end{proposition}
\begin{proof} {\it (Common)}
The proofs of \refP{mcocinf} and \refP{mcoc0} can be done in a completely
analogous way as for the vector field algebra. But things can be simplified
by using the results for the function and vector field algebra and the
simple fact $A_n'=nB_{n-1}$.
Hence for $C$ any of the elements we get
\begin{equation}
C\cdot A_n''=n\, C\cdot B_{n-1}',
\quad
C\cdot B_{n}''=\frac 1{n+1}C\cdot A_{n+1}'''\, ,n\ne -1.
\end{equation}
Consequently,  ($a=0,\infty$)
\begin{equation}
\begin{gathered}
\gm[a](e_n,A_m)=m\cdot \gA[a](A_{n+1},B_{m-1}),\qquad
\gm[a](f_n,A_m)=m\cdot \gA[a](B_{n+1},B_{m-1}),
\\
\gm[a](e_n,B_m)=(1/(m+1))\cdot \gL[a](e_{n},e_{m}),\qquad
\gm[a](f_n,B_m)=(1/(m+1))\cdot \gL[a](f_{n},e_{m}).
\end{gathered}
\end{equation}
For the second line we exclude $m=-1$.
By simple substitution we obtain exactly the claimed  expressions.
In fact even for $m=-1$ the results obtained by substitution are correct, but
of course need to be verified directly.
\end{proof}

As shown in \refS{central} the differential operator 
algebra admits a universal central extension and the introduced
six geometric cocycles, each associated to a different central 
element will yield the universal central extension.

\begin{proposition}
A cocycle class $[\ga]$ for $\Do$ will be local
(with respect to the standard splitting)
if and only if $\ga$ is cohomologous to a linear combination
\begin{equation}\label{eq:summ}
\ga\quad\sim\quad
\alpha_1\cdot \gA[\infty]
+
\alpha_2\cdot \gL[\infty]
+
\alpha_3\cdot \gm[\infty],\qquad \alpha_i\in\C.
\end{equation}
In particular, the space of local cohomology classes is 
3-dimensional. 
\end{proposition}
This is a general result of \cite{Schlloc}, \cite{Schlkn} which
is for the 3-point illustrated by the above calculations.
In the very general case (meaning arbitrary genus, arbitrary
number of marked points, arbitrary splitting)
the three cocycles in \refE{summ} are obtained
by integrating over a separating cycle.
By allowing suitable meromorphic affine and projective connections 
holomorphic outside of $A$ we can even obtain equality 
in \refE{summ}.

\section{Lie superalgebras}
\label{sec:super}


Lie superalgebras of Krichever--Novikov type were discussed
in  \cite{Schlsa}, \cite{Schlkn}, see also Kreusch \cite{Kreusch}
and Leidwanger and Morier-Genoud \cite{LeiMor}.
Here we 
present the $N$-point genus zero situation. In particular
we will also deal with the $N=3$ case.

A Lie superalgebra $\Sa$ is 
a vector space which is decomposed into  subspaces of 
even and odd elements
$\Sa=\Sa_{\bar 0}\oplus\Sa_{\bar 1}$,
i.e. $\Sa$ is a $\Z/2\Z$-graded vector space.
Furthermore, let $[.,.]$ be a 
  $\Z/2\Z$-graded bilinear map  $\Sa\times \Sa\to \Sa$ such
that for elements $x,y$ of pure parity
\begin{equation}\label{eq:ssup}
[x,y]=-(-1)^{\bar x \bar y}[y,x].
\end{equation}
Here $\bar x$ is the parity of $x$, etc.
These conditions say that
\begin{equation}\label{eq:szg}
[\Sa_{\bar 0},\Sa_{\bar 0}]\subseteq \Sa_{\bar 0},\qquad
[\Sa_{\bar 0},\Sa_{\bar 1}]\subseteq \Sa_{\bar 1},\qquad
[\Sa_{\bar 1},\Sa_{\bar 1}]\subseteq \Sa_{\bar 0},
\end{equation}
and that $[x,y]$ is symmetric for $x$ and $y$ odd, otherwise
anti-symmetric.
Now $\Sa$ is a \emph{Lie superalgebra}
\index{Lie superalgebra}
 if in addition
the 
\emph{super-Jacobi identity}
\index{super-Jacobi identity}
  (for $x,y,z$ of pure parity)
\begin{equation}\label{eq:jacsup}
(-1)^{\bar x \bar z}[x,[y,z]]+
(-1)^{\bar y \bar x}[y,[z,x]]+
(-1)^{\bar z \bar y}[z,[x,y]]
=0
\end{equation}
is valid.

The Lie superalgebra, we 
discuss here, is 
based on  half-forms.
In fact, it will be a superalgebra of Neveu-Schwarz type.
Recall the associative product
\refE{aprod}
\begin{equation}
\sbul\ \Fl[-1/2]\times \Fl[-1/2]\to \Fl[-1]=\La\;.
\end{equation}
We introduce the vector space $\Sa$ with  the product
\begin{equation}\label{eq:saprod}
\Sa:=\La\oplus\Fl[-1/2],\quad
[(e,\varphi),(f,\psi)]:=([e,f]+\varphi\;\sbul\;\psi,
e\ldot \varphi-f\ldot \psi).
\end{equation}
The elements of $\La$ are denoted by $e,f, \dots$, and the elements of
$\Sp$ by $\varphi,\psi,\ldots$.

The definition \refE{saprod} can be reformulated as an extension
of $[.,.]$ on  $\La$ to a super-bracket (denoted by the same
symbol) on $\Sa$ by setting
\begin{equation}\label{eq:sadef}
[e,\varphi]:=-[\varphi,e]:=e\ldot \varphi
=(e\frac {d\varphi}{dz}-\frac 12\varphi\frac {de}{dz})(dz)^{-1/2}
\end{equation}
and
\begin{equation}
[\varphi,\psi]:=\varphi\;\sbul\; \psi\; .
\end{equation}
The elements of $\La$ are the elements of even parity,
and the elements of $\Sp$ the elements of odd parity.
\begin{proposition}\label{P:knsuper}
\cite[Prop.~2.15]{Schlkn}
The space $\Sa$ with the above introduced parity and product
is a Lie superalgebra.
\end{proposition}
Clearly, the vector field algebra $\La$ is a Lie subalgebra.
\begin{remark}
Recall that 
for genus $g\ge 1$ the choice of a square root of the
canonical bundle, also called 
a theta characteristics, will be needed for
fixing the super algebra. Such a choice 
 corresponds to choosing
a 
\emph{spin structure}
 on $\Sigma_g$.
\end{remark}

Central extensions of $\Sa$ are also of relevance and 
will be given in the following. The Lie algebra 
2-cocycle conditions for $\ga$ have to be replaced by its super
version,
saying 
\begin{equation}\label{eq:csymm}
\ga(x,y)=-(-1)^{\bar x \bar y}\ga(x,y).
\end{equation}
and
\begin{equation}\label{eq:scocyc}
(-1)^{\bar x \bar z}\ga(x,[y,z])+
(-1)^{\bar y \bar x}\ga(y,[z,x])+
(-1)^{\bar z \bar y}\ga(z,[x,y])\ = \ 0.
\end{equation}
As usual our elements $x,y,z$ should be of pure parity.

\begin{proposition}\cite{Schlsa}
Let $C$ be any closed (differentiable) curve on $\Sigma_g$ not meeting
the
points in $A$, and let $R$ be any (holomorphic) projective connection
then the bilinear extension of 
\begin{equation}
\label{eq:sacoc}
\begin{aligned}
\gS[{C,R}](e,f)&:=\cint{C} \left(\frac 12(ef'''-e'''f)
-R\cdot(ef'-e'f)\right)dz
\\
\gS[{C,R}](\varphi,\psi)
&:=
\cint{C} \left(\varphi''\cdot \psi+\varphi\cdot \psi''
-R\cdot\varphi\cdot \psi\right)dz
\\
\gS[{C,R}](e,\varphi)&:=0
\end{aligned}
\end{equation}
gives a geometric Lie superalgebra cocycle for $\Sa$, hence defines a central
extension of $\Sa$.
A different projective connection will yield a cohomologous 
cocycle.
\end{proposition}
A similar formula was given  by
Bryant in \cite{Bryant}. By adding the projective connection in
the second part of 
\refE{sacoc} he corrected some formula appearing in 
\cite{BMRR}. He only considered the two-point case and only the
integration over a separating cycle.
See also \cite{Kreusch} for the multi-point case, where  still only
the integration over a separating cycle is considered.

In contrast to the differential operator algebra case the two parts 
cannot be prescribed independently. Only with the same integration path
(more precisely, homology class) and the given factors in front
of the integral it will work.

\bigskip

Now we consider the genus zero case. 
As shown above, all cohomology classes for the vector field algebra
$\La$ are bounded classes and are obtained via integrating over $C_i$,
hence by calculating residues. From the general theorem
\cite{Schlsa}, \cite{Schlkn} it follows that the cohomology
classes for the superalgebra will be uniquely be given by
taking the same point for calculating the residue.

For explicite calculations of the structure constants see
\refP{lmod} and \refP{aprod}.
We will only give the results for the 3-point case using the notation 
of \refS{3point}. 
In addition to the basis elements $e_n$ and  $f_n$ of $\L$ we take
\begin{equation}
\varphi_n=A_{n+1/2}(dz)^{-1/2},\quad
\psi_n=B_{n+1/2}(dz)^{-1/2},\qquad
n\in\Z+1/2.
\end{equation}
Additionally to the structure constants of \refP{vprod}
we have 
\begin{proposition}\label{P:superprod}
\begin{equation*}
\begin{aligned}{}
[\varphi_n, \varphi_m]&=e_{n+m}
\\
[\varphi_n, \psi_m]&=f_{n+m}
\\
[\psi_n, \psi_m]&=e_{n+m}+4\,e_{n+m+1}
\\
[e_n, \varphi_m]&=(m-n/2)\,\psi_{n+m}
\\
[e_n, \psi_m]&=(m- n/2)\,\varphi_{n+m}+(4m-2n+2)\,\varphi_{n+m+1}
\\
[f_n, \varphi_m]&=(m-n/2)\,\varphi_{n+m}+(4m-2n-1)\,\varphi_{n+m+1}
\\
[f_n, \psi_m]&=(m-n/2)\,\psi_{n+m}+(4m-2n+1)\,\psi_{n+m+1}\, .
\end{aligned}
\end{equation*}
\end{proposition}
Similar expressions are given by Leidwanger and Morier-Genoud
\cite[Prop.~3.8]{LeiMor} with respect to a 
slightly different scaled system of basis elements also used in
\cite{SchlDeg}, \cite{FiaSchl1}.

Next we consider 2-cocycles.
Again with respect to the standard system of coordinates
$(z,w=1/z)$ the connection $R=0$ can be chosen.
We have a two-dimensional space of geometric
cocycles generated by 
$\gS[{0}]$ and 
$\gS[{\infty}]$ obtained by taking the residue from
\refE{sacoc} at $0$ and $\infty$.
For pairs of pure vector field type arguments we have the result
of \refP{lcocinf} and  \refP{lcoc0}.
For mixing of pure types it is zero. It remains to consider 
pairs of $-1/2$-forms.
As
\begin{equation*}
\res_a(\varphi''\psi+\varphi\psi'')
=
\res_a(\varphi''\psi+\varphi\psi''-(\varphi\psi')'+(\varphi'\psi)')
=2\res_a(\varphi''\psi)=2\res_a(\psi''\varphi)
\end{equation*}
we have
\begin{equation}
\begin{aligned}{}
\gS[\infty](\varphi,\psi)&=
\res_\infty((\varphi''\psi+\varphi\psi'')\,dz)
=2\res_\infty(\varphi''\psi\, dz)
\\
\gS[0](\varphi,\psi)&=
\res_0((\varphi''\psi+\varphi\psi'')\,dz)
=2\res_0(\varphi''\psi\, dz).
\end{aligned}
\end{equation}
Note that these values are already calculated
in the context for the mixed cocycle for the
differential operator algebra. In fact we have
with $\hat n=n+1/2$ and $\hat m=m+1/2$
\begin{equation}
\begin{gathered}
\gS[a](\varphi_n,\varphi_m)
=2\gm[a](e_{\hat n-1},A_{\hat m}),
\qquad
\gS[a](\varphi_n,\psi_m)
=2\gm[a](e_{\hat n-1},B_{\hat m}),
\\
\gS[a](\psi_n,\varphi_m)
=2\gm[a](f_{\hat n-1},A_{\hat m}),
\qquad
\gS[a](\psi_n,\psi_m)
=2\gm[a](f_{\hat n-1},B_{\hat m})\ .
\end{gathered}
\end{equation}
With these identities we recall from the results there
\begin{proposition}
\label{P:scocinf}
\begin{equation*}
\begin{aligned}{}
\gS[\infty](\varphi_n,\varphi_m)&=0
\\
\gS[\infty](\varphi_n,\psi_m)&=
-4(n-1/2)(n+1/2)\delta_m^{-n}
-8n(2n+1)\delta_{m}^{-n-1}.
\\
\gS[\infty](\psi_n,\psi_m)&=0
\end{aligned}
\end{equation*}
\end{proposition}
\begin{proposition}
\label{P:sco0}
\begin{equation*}
\begin{aligned}{}
\gS[0](\varphi_n,\varphi_m)&=
-2(n+1/2)(n-1/2)\delta_m^{-n}
+4(n+1/2)^2\delta_m^{-n-1}
\\
&\qquad
+2\sum_{k\ge 2}(n+1/2)(n-1/2+k)(-1)^k2^k
\frac {(2k-3)!!}{k!}\delta_m^{-n-k}
\\
& 
\\
\gS[0](\varphi_n,\psi_m)&=
2(n-1/2)(n+1/2)\delta_m^{-n}
+4n(2n+1)\delta_{m}^{-n-1}.
\\
&
\\
\gS[0](\psi_n,\psi_m)&=
-2(n+1/2)(n-1/2)\delta_m^{-n}
-12(n+1/2)^2\delta_m^{-n-1}
-12(n+1/2)(n+3/2)\delta_m^{-n-2}
\\
&\qquad
+2\sum_{k\ge 3}(n+1/2)(n-1/2+k)(-1)^{k-1}2^k\cdot 3\cdot
\frac {(2k-5)!!}{k!}\delta_m^{-n-k}\;.
\end{aligned}
\end{equation*}
\end{proposition}
We illustrated again that only the cocycle $\gS[\infty]$ is local with
respect to the standard splitting. The $\gS[\infty]$ is the one
which was considered by Kreusch \cite{Kreusch} (up to a different
indexing of the basis elements).
\begin{remark}
Here  we considered the central element to be
even. We could have even dropped this assumption.
In \cite{Schlsa} we show that 
the corresponding cocycles are cohomologically
trivial.
\end{remark}
\begin{remark}
Leidwanger and Morier-Genoux introduced in \cite{LeiMor}  a
\emph{Jordan superalgebra}
 in the  geometric  setting described here. They put 
\begin{equation}
\mathcal{J}:=\A\oplus\Fl[-1/2]=
\mathcal{J}_{\bar 0}\oplus
\mathcal{J}_{\bar 1}.
\end{equation}
and  define the (Jordan) product $\circ$   via the algebra structures 
 for the spaces $\Fl$ by
\begin{equation}
\begin{aligned}[]
f\circ g&:=f\;\sbul\; g\quad \in\A,
\\
f\circ\varphi&:=f\;\sbul\; \varphi \quad \in\Fl[-1/2],
\\
\varphi\circ\psi&:=[\varphi,\psi]\quad \in\Fl[0].
\end{aligned}
\end{equation}
It is a non-associative extension of 
the associative 
algebra $A$ of 
meromorphic functions.
By rescaling the second definition with the factor 1/2 one obtains
a 
\emph{Lie anti-algebra}
 as introduced by
Ovsienko \cite{Ovs1}. 

Of  course it is easy  again to 
express everything in terms of our introduced 
Krichever--Novikov
type basis 
and in particular to calculate the structure equations in
the genus zero case completely in the same manner as above.
We calculate
\begin{equation}
\begin{gathered}
A_n\circ A_m=A_{n+m},\quad
A_n\circ B_m=B_{n+m},\quad
B_n\circ B_m=A_{n+m}+4A_{n+m+1},
\\
A_n\circ\varphi_m=\varphi_{n+m},\quad
A_n\circ\psi_m=\psi_{n+m},\quad
B_n\circ\varphi_m=\psi_{n+m},\quad
B_n\circ\psi_m=\varphi_{n+m}+4\varphi_{n+m},
\end{gathered}
\end{equation}
and  using \refE{aliea}
\begin{equation}
\begin{aligned}{}
\varphi_n\circ\varphi_m&= 1/2(m-n)B_{n+m},
\\
\varphi_n\circ\psi_m&= 1/2(m-n)A_{n+m}+(2(m-n)+1)A_{n+m+1}
=-\psi_m\circ\varphi_n,
\\
\psi_n\circ\psi_m&= 1/2(m-n)B_{n+m}+2(m-n)B_{n+m+1}\,.
\end{aligned}
\end{equation}
See also \cite{LeiMor}.
The almost-graded structure is obvious.

\end{remark}

\section{Remarks on representations}
\label{sec:further}

Having recognized the genus zero algebras as 
examples of multi-point algebras of KN type,
complete
collection of their representations as 
presented e.g. in \cite{Schlkn} can be studied.
In this article we will not carry this out, but just name them
with a few hints.

We start from the natural action of our vector fields, functions, 
(and  more general current Lie algebras) on forms of  weight $\la$.
These representations are not of the type one is looking for, e.g.
in physics. They do not have a ground state (a vacuum), no creation
operators,
no annihilation operators. But after choosing a splitting with
an induced almost-grading and adapted basis one can construct
semi-infinite wedge forms of weight $\la$.
They supply candidates for such desired representations. To extend the
natural representation to the wedge-forms we have to regularize
the action. The resulting action will only be a projective
Lie action. The cocycle defining it defines a central extensions of
the algebras under consideration. It turns out that the cocycle is
local and as a consequence the central extension is almost-graded.
Of course the reference is always the almost-grading induced by the 
splitting from which we started.
We could e.g. start with the standard splitting 
and obtain exactly the one-dimensional central extensions identified 
in this article  as 
allowing the extension of the almost-grading, i.e. the ones obtained
by taking the residue at $\infty$.

For $N>2$ we have more than one splitting and hence more than one
almost-grading. For each splitting we will obtain another representation and
extension of the original algebra.  

Via this process we get  semi-infinite
wedge representations, or equivalently fermionic Fock space 
representations of a centrally extended 
vector field algebra, differential operator algebra, respectively
affine Lie algebra
\cite[Sect.~7]{Schlkn}.
Furthermore we have $b-c$-systems,
and fields in CFT \cite[Sect.~8]{Schlkn}.

As far as the affine Lie algebras are concerned, Verma modules and
highest weight representations are given 
\cite[Sect.~9.9]{Schlkn}.
Again without fixing an almost-grading we cannot 
even talk about highest weight
representations.
Of special relevance is the Sugawara representation
(energy-momentum representation). 
Let $\gh$ be the affine Lie algebra associated to 
a simple Lie algebra $\g$ or the Heisenberg algebra.
In both cases the defining cocycle for the central extension should be local.
Given an ``admissible representation'', e.g. a highest weight representation,
the arbitrary  genus, multi-point 
Sugawara construction 
done by the author together with Sheinman 
\cite{SSS}
can applied
(see \cite[Sect.~10]{Schlkn} for some
simplifications). 
The level $c$ of the representation is the scalar by which
the central element of the affine algebra acts.
Let $\kappa$ be the dual Coxeter number, respectively
$\kappa=0$ for the Heisenberg algebra. If the level is
non-critical, e.g. if $c+\kappa\ne 0$ then
the Sugawara operators can be rescaled and the rescaled
operators yield a  representation of the centrally
extended vector field algebra given by 
a cocycle which is local with respect to the 
almost-grading we started with.
In fact, via the Sugawara operator we obtain a
projective representation of $\Dg$ which is the semi-direct 
product of $\g$ with $\La$
discussed in \refS{gdiff}.
By passing to a central
extension we get a honest Lie representation of 
$\Dhg$, see \cite[Prop.~10.15]{Schlkn}.
The Sugawara representations appear in the
context of WZNW models, see  e.g. 
\cite{SSpt}, \cite{SSpt2}, \cite[Sect.~11]{Schlkn},
\cite{SheBo} for  arbitrary genus. They are also
of relevance in genus zero.
Some steps in directions of vertex operator algebras 
are done by Linde \cite{Lind1}, \cite{Lind2}.

The choice of an almost-grading is an overarching and necessary
concept in the theory of representations of our algebras.
Note that in genus zero and two points it is always
implicitly given by the standard grading (and it is 
a grading) which allows the representation theory to
get of the grounds.
For the general situation the choice of an
almost-grading does the job. 
Such an almost-grading always exists.
But  we have
to take into account that we will have finitely many choices as we
will have a finite number of different splitting and hence almost-gradings
for the same algebra.
{}From the point of the non-extended algebra they 
correspond to 
different projective representations.

\medskip

Another point which can be addressed is the
symmetry aspect.
In the genus zero situation for $N=3$ we have 
additional geometric symmetries which induces automorphism of 
the algebras (even after fixing the three points).
For the generic $N=4$ case there are no such 
additional symmetries.
But for special choices of these points there might be
some inducing also 
automorphism of the algebras.
Klein's list of possible finite subgroups of
 $\PGL(2,\C)$
given by 
$C_N,D_N,A_4, A_5, S_4$
show up.
See \cite{Lom1}, \cite{Lom2}, \cite{Lom2}, 
\cite{CGLZ}.
For an approach via KN objects, see \cite{Chopp}.

\appendix
\section{Some useful formulas for the three-point case}
\label{sec:calc}
\subsection{Definitions}
Our points where poles are allowed are normalized 
to be $z=0$, $z=1$ and $z=\infty$. 
We define the following basic functions
admitting only poles there.
\begin{equation}\label{eq:abd}
\begin{aligned}
A_n(z)&:=z^n(z-1)^n,
\\
B_n(z)&:=z^n(z-1)^n(2z-1)=A_n(z)\cdot (2z-1).
\end{aligned}
\end{equation}
Note that $A_0=1$ and $B_0=2z-1$, and  that $A_n$ and $B_n$ are holomorphic
outside of $\infty$ if and only if $n\ge 0$. More precisely,
\begin{lemma}\label{L:abdiv}
The divisors corresponding to $A_n$ and $B_n$ are
\begin{equation*}
\begin{aligned}
(A_n)&=n\cdot[0]+n\cdot[1]-2n\cdot[\infty],
\\
(B_n)&=n\cdot[0]+n\cdot[1]+1\cdot [1/2]-(2n+1)\cdot[\infty].
\end{aligned}
\end{equation*}
\end{lemma}
\begin{proof}
For finite $z$ values this is obvious. For the order at the
point $\infty$ we have to replace $z$ by $1/w$ with $w$ the local
variable at $\infty$. For example we obtain
\begin{equation}
A_n(z(w))=w^{-n}(1/w-1)^{n}=w^{-2n}(1-w)^n.
\end{equation}
Hence, the statement. Accordingly we get the result for $B_n$.
\end{proof}
The following is obvious.
\begin{lemma}\label{L:absym} We have the symmetry
\begin{equation*}
A_n(1-z)=A_n(z),\qquad
B_n(1-z)=-B_n(z)\ .
\end{equation*}
\end{lemma}
\subsection{Products}
\begin{lemma}\label{L:prod}
\begin{equation*}
\begin{aligned}
A_n\cdot A_m&=A_{n+m},
\\
A_n\cdot B_m&=B_{n+m},
\\
B_n\cdot B_m&=A_{n+m}+4A_{n+m+1}.
\end{aligned}
\end{equation*}
\end{lemma}
\begin{proof}
The first two relations are by the 
definitions of $A_n$ and $B_n$.
The last follows from 
\begin{equation}\label{eq:qwf}
(2z-1)^2=1+4z(z-1).
\end{equation}
\end{proof}
\subsection{Derivatives}
For the Lie product and for the Lie algebra cocycles we will need
the derivatives of our basis functions.
They will be linear combinations of the $A_n$ and $B_n$. 
We will need their explicit expressions.
\begin{lemma}\label{L:deri}
For the derivatives of our basic functions we have
\begin{equation*}
\begin{aligned}
A_n'&=n\,B_{n-1},
\\
B_n'&=2(2n+1)\,A_n+n\,A_{n-1},
\\
A_n''&=2n(2n-1)\,A_{n-1}+n(n-1)\,A_{n-2},\\
B_n''&=2n(2n+1)\,B_{n-1}+n(n-1)\,B_{n-2},
\\
A_n'''&=2n(2n-1)(n-1)\,B_{n-2}+n(n-1)(n-2)\,B_{n-3},\\
B_n'''&=4n(2n+1)(2n-1)\,A_{n-1}\\
\qquad &\qquad
+4n(n-1)(2n-1)\,A_{n-2}
+n(n-1)(n-2)\,A_{n-3}\, .
\end{aligned}
\end{equation*}
\end{lemma}
\begin{proof}
Starting from $A_n(z)=z^n(z-1)^n$ we obtain
\begin{equation*}
A_n'(z)=n\,z^{n-1}(z-1)^{n-1}(2z-1)=n\,B_{n-1}(z).
\end{equation*}
Furthermore,
\begin{multline*}
B_n'=(A_n\cdot(2z-1))'=
nB_{n-1}\cdot(2z-1)+A_n\cdot 2
\\
=n\,B_{n-1}\cdot B_{0}+2A_n
=2(2n+1)A_n+nA_{n-1}.
\end{multline*}
For this  we used \refL{prod}.
These are the basic results. 
Now we use them to calculate the higher
order derivative.
For example $A_n''=nB_{n-1}'$ and by substituting 
$B_{n-1}'$ we get the expression in the lemma.
By the same way all other results can be calculated.
\end{proof}
Using \refL{deri} and \refL{prod} we 
immediately verify the following relations which will be needed
in the main text.
\begin{lemma}\label{L:a1prim}
\begin{equation*}
\begin{aligned}
A_nA_m'&=m\,B_{n+m-1},
\\
A_nB_m'&=2(2m+1)\,A_{n+m}+m\,A_{n+m-1},
\\
B_nA_m'&=4m\,A_{n+m}+m\,A_{n+m-1},
\\
B_nB_m'&=2(2m+1)\,B_{n+m}+m\,B_{n+m-1}.
\end{aligned}
\end{equation*}
\end{lemma}
\begin{lemma}\label{L:a2prim}
\begin{equation*}
\begin{aligned}
A_nA_m''&=2m(2m-1)\,A_{n+m-1}+m(m-1)\,A_{n+m-2},
\\
A_nB_m''&=2m(2m+1)\,B_{n+m-1}+m(m-1)\,B_{n+m-2},
\\
B_nA_m''&=2m(2m-1)\,B_{n+m-1}+m(m-1)\,B_{n+m-2},
\\
B_nB_m''&=8m(2m+1)\,A_{n+m}+2m(4m-1)\,A_{n+m-1}+m(m-1)\,A_{n+m-2}.
\end{aligned}
\end{equation*}
\end{lemma}
\begin{lemma}\label{L:a3prim}
\begin{equation*}
\begin{aligned}
A_nA_m'''&=2m(2m-1)(m-1)\,B_{n+m-2}+m(m-1)(m-2)\,B_{n+m-3},
\\[2pt]
A_nB_m'''&=4m(2m+1)(2m-1)\,A_{n+m-1}+
4m(m-1)(2m-1)\,A_{n+m-2}
\\
&\qquad
+
m(m-1)(m-2)\,A_{n+m-3},
\\[2pt]
B_nA_m'''&=8m(2m-1)(m-1)\,A_{n+m-1}+
2m(m-1)(4m-5)\,A_{n+m-2}
\\
&\qquad
+
m(m-1)(m-2)\,A_{n+m-3},
\\[2pt]
B_nB_m'''&=4m(2m+1)(2m-1)\,B_{n+m-1}+
4m(2m-1)(m-1)\,B_{n+m-2}
\\
&\qquad
+
m(m-1)(m-2)\,B_{n+m-3}.
\end{aligned}
\end{equation*}
\end{lemma}

\subsection{Residues}
Next we calculate the residues of $A_ndz$ and $B_ndz$
at the points where poles might be.
Recall that these are the points $0,1$ and $\infty$.

For this aim we need the Laurent series expansion of $(z-1)^m$
around zero.
We collect the following well-known facts about binomial
series. 

The expansion
\begin{equation}\label{eq:bf1}
(z-1)^m=\sum_{k=0}^\infty
\binom {m}{k}\,z^k(-1)^{m-k},\qquad z\in\C,\ |z|<1
\end{equation}
is valid for all $m\in\Z$.

For negative exponents an equivalent
expression is 
\begin{equation}\label{eq:bf2}
(z-1)^{-n}=(-1)^n
\sum_{k=0}^\infty 
\binom {n+k-1}{n-1}\,z^k, \qquad 
z\in\C,\ 
|z|<1,
\end{equation}
where $n\in\N$.
Later we will need the easy relation
\begin{equation}\label{eq:Bh}
\binom {2k}{k}=\frac {(2k-1)!!}{k!}\;2^k,\qquad k\in\N,
\end{equation}
where $(2k-1)!!=1\cdot 3\cdots (2k-1)$ is the double factorial.
\begin{lemma}\label{L:amres}
For the residues at the point $0$ we have
\begin{equation*}
\res_{0}(A_{-n}dz)=
\begin{cases} 0,&n\le 0,
\\
-1, &n=1,
\\
(-1)^n\dfrac {(2n-3)!!}{(n-1)!}\,2^{n-1},
&n\ge -2.
\end{cases}
\end{equation*}
\end{lemma}
\begin{proof}
If $n<0$ then there is no pole at $z=0$. Hence let $n>0$. 
 We use \refE{bf2} and calculate
\begin{equation}
A_{-n}(z)=(-1)^n\sum_{k=0}^\infty 
\binom {n+k-1}{n-1}z^{k-n}.
\end{equation}
The residue at $z=0$ is given by the coefficient paired with
$z^{-1}$. Hence it is given by the coefficient  for $k=n-1$
\begin{equation}
\res_{0}(A_{-n}dz)=(-1)^n\binom {2(n-1)}{n-1}.
\end{equation}
For $n=1$ we obtain the value $-1$. For $n>1$ we use 
\refE{Bh} and obtain
\begin{equation} 
\res_{0}(A_{-n}dz)
=(-1)^n\frac {(2n-3)!!}{(n-1)!}\,2^{n-1}.
\end{equation}
This was the claim.
\end{proof}
\begin{lemma}\label{L:bmres}
\begin{equation*}
\res_{0}(B_mdz)=
\begin{cases} 
1, &m=-1,
\\
0,
&\text{otherwise}.
\end{cases}
\end{equation*}
\end{lemma}
\begin{proof}
{}From \refL{deri} we conclude
$B_m=(1/(m+1))A_{m+1}$ if $m\ne -1$. Hence, for $m\ne -1$ 
the differential $B_mdz$ is an exact differential and hence does not
have any residue.
It remains 
$B_{-1}(z)=z^{-1}(z-1)^{-1}(2z-1)$ which obviously has
as residue $+1$ at $z=0$.
\end{proof}
%
\begin{lemma}\label{L:res1}
\begin{equation*}
\begin{aligned}
\res_{1}(A_mdz)&=-\res_{0}(A_mdz)
\\
\res_{1}(B_mdz)&=\res_{0}(B_mdz)\;.
\end{aligned}
\end{equation*}
\end{lemma}
\begin{proof}
We make a change of local coordinates $v=1-z$. In particular we have
$dv=-dz$. 
Hence, for the local representing function $\tilde A_m(v)$ we have 
\begin{equation*}
\tilde A_m(v)=-A_m(z(v))=-A_m(1-v)=-A(v),
\end{equation*}
and $\tilde B_m(v)=B_m(v)$ by using 
\refL{absym}.
Moreover, the point $z=1$ corresponds to the point $v=0$. As the 
residue is independent of the choice of local coordinates we obtain
exactly the statement of the lemma.
\end{proof} 
\begin{lemma}\label{L:resinf}
\begin{equation*}
\begin{aligned}
\res_{\infty}(A_mdz)&=0,
\\
\res_{\infty}(B_mdz)&=-2\res_{0}(B_mdz),
\\
&=\quad\begin{cases}
-2,&m=-1,
\\
0,&\text{otherwise}.
\end{cases}
\end{aligned}
\end{equation*}
\end{lemma}
\begin{proof}
By the residue theorem \cite{SchlRS} for a compact
Riemann surface the sum over all residues of a meromorphic
differential is zero. As our differentials have only poles
at $0,1$, $\infty$ the claim follows from
Lemmas \ref{L:amres},  \ref{L:bmres}, and  \ref{L:res1}.
\end{proof}
\begin{remark}
The proofs of Lemmas \ref{L:amres} and
  \ref{L:bmres} do not really make reference to the
$N=3$ situation. But for the proofs of the
Lemmas \ref{L:res1} and \ref{L:resinf} the symmetry 
for $N=3$ is
used.
Nevertheless, the principal statement in \refL{resinf},
telling us that only for finitely many $m$ there might
be a non-vanishing residue, remains true. To indicate this I 
present an alternative proof  of \refL{resinf} which generalizes
to arbitrary $N$.
\begin{proof}
We make the  change of coordinates
\begin{equation}
z=\frac 1w,\qquad dz=-\frac 1{w^2}dw
\end{equation}
and express the elements $A_m(z)dz$ and $B_m(z)dz$ in the
new coordinates. In the first case we obtain
\begin{equation}
-w^{-2m-2}(1-w)^mdw,
\end{equation}
in the second case
\begin{equation}
-w^{-2m-3}(1-w)^m(2-w)dw.
\end{equation}
For the existence of 
a pole at $\infty$ it is necessary that 
$-2m-2<0$ respectively 
$-2m-3<0$. 
Hence, $m>-1$, respectively $m\ge-1$.
If $m\ge 0$ there are no poles at the other points, hence
by the residue theorem there is no residue at $\infty$.
This says that in the first case there is no residue possible 
at all and
in the second case only for $m=-1$. It calculates as
\begin{equation}
-\res_{\infty}(w^{-1}(1-w)^{-1}(2-w)dw)=-2.
\end{equation}
\end{proof}
\end{remark}


\section{Three-point $\sln(2,\C)$-current algebra for genus 0}
\label{sec:sl2}
In this appendix we will give the example 
of the universal central extension of the
3-point $\sln(2,\C)$-current algebra. The general theory has been
developed in \refS{aff}. The 3-point  $\slnb(2,\C)$ algebra is of relevance
in quite a number of applications, we only name statistical mechanics
\cite{HaTe}, \cite{ItTe}.
Hence, the explicite knowledge of the structure equations
with respect to some basis
might be of some interest. We take 
$\sln(2,C)$ with the standard matrix generators
\begin{equation}
H:=\begin{pmatrix} 1&0\\ 0&-1 
\end{pmatrix}, \quad
X:=\begin{pmatrix} 0&1\\ 0&0 
\end{pmatrix}, \quad
Y:=\begin{pmatrix} 0&0\\ 1&0 
\end{pmatrix},
\end{equation}
and induced relations
\begin{equation}
[H,X]=2X,\quad
[H,Y]=-2Y,\quad
[X,Y]=2X\, .
\end{equation}
As basis elements for the current algebra $\slnb(2,\C)$ with 
respect to the almost-grading introduced in \refS{3point} we take
the elements
\begin{equation}
Z^{(s)}=Z\otimes A_n,\quad
Z^{(a)}=Z\otimes B_n,\qquad
Z\in\{H,X,Y\}.
\end{equation}
The structure of the current algebra comes via \refE{currint}
from $\A$ (and of course from $\g$). 
We need the Cartan-Killing form. Up to a normalization 
it is given by
\begin{equation}
\beta(A,B)=\tr(A\cdot B).
\end{equation}
Hence,
\begin{equation}
\beta(X,Y)=\beta(Y,X)=1,\quad
\beta(H,H)=2,\quad
\beta(H,X)=\beta(H,Y)=\beta(X,H)=\beta(Y,H)=0.
\end{equation}
The universal central extension will have a two-dimensional
center and will be given by
\begin{equation}
[Z\otimes f,W\otimes g]=
[Z,W]\otimes f\cdot g
+\alpha_\infty\cdot\beta(Z,W)\cdot\gA[\infty](f,g)\cdot t_\infty
+\alpha_0\cdot\beta(Z,W)\cdot\gA[0](f,g)\cdot t_0,
\end{equation}
with $t_\infty, t_0$ central elements, $\alpha_\infty,\alpha_0\in\C$.
Recall that $\gA[a](f,g)$ can be calculated as $\res_a(fdg)$.

{}From the general theory developed in \refS{aff} we know that
the central extension will be almost-graded 
with respect to the standard splitting if and only if
$\alpha_0=0$.

All the data needed has been calculated already before. If we
collect them we obtain the following results.
\begin{equation*}
\begin{aligned}{}
[H_n^{(s)},H_m^{(s)}]&=\alpha_\infty\cdot 4n\cdot \delta_m^{-n}\cdot
t_\infty
-\alpha_0\cdot 2n\cdot \delta_m^{-n}\cdot
t_0,
\\
[H_n^{(s)},H_m^{(a)}]&=
2\alpha_0\bigg(
n\,\delta_m^{-n}
+2n\,\delta_m^{-n-1}
+\sum_{k=2}^\infty n\,(-1)^{k-1}
2^k\frac{(2k-3)!!}{k!}\delta_{m}^{-n-k}\,
\bigg)\cdot t_0.
\\
[H_n^{(a)},H_m^{(a)}]&=
\alpha_\infty(4n\delta_m^{-n}+8(2n+1)\delta_m^{-n-1})\cdot t_\infty
-\alpha_0(2n\delta_m^{-n}+4(2n+1)\delta_m^{-n-1})\cdot t_0
\end{aligned}
\end{equation*}
\begin{equation*}
\begin{gathered}{}
[H_n^{(s)},X_m^{(s)}]=2X_{n+m}^{(s)},\quad
[H_n^{(s)},X_m^{(a)}]=2X_{n+m}^{(a)},\quad
[H_n^{(a)},X_m^{(a)}]=2X_{n+m}^{(s)}+8X_{n+m}^{(s)},
\\
[H_n^{(s)},Y_m^{(s)}]=-2Y_{n+m}^{(s)},\quad
[H_n^{(s)},Y_m^{(a)}]=-2Y_{n+m}^{(a)},\quad
[H_n^{(a)},Y_m^{(a)}]=-2Y_{n+m}^{(s)}-8Y_{n+m}^{(s)},
\end{gathered}
\end{equation*}
\begin{equation*}
\begin{aligned}{}
[X_n^{(s)},Y_m^{(s)}]&=H_{n+m}^{(s)}+
\alpha_\infty\cdot 2n\cdot \delta_m^{-n}\cdot
t_\infty
-\alpha_0\cdot n\cdot \delta_m^{-n}\cdot
t_0,
\\
[X_n^{(s)},Y_m^{(a)}]&=H_{n+m}^{(a)}+
\alpha_0\bigg(
n\,\delta_m^{-n}
+2n\,\delta_m^{-n-1}
+\sum_{k=2}^\infty n\,(-1)^{k-1}
2^k\frac{(2k-3)!!}{k!}\delta_{m}^{-n-k}\,
\bigg)\cdot t_0.
\\
[X_n^{(a)},Y_m^{(a)}]&=H_{n+m}^{(s)}+4H_{n+m+1}^{(s)}+
\alpha_\infty(2n\delta_m^{-n}+4(2n+1)\delta_m^{-n-1})\cdot t_\infty
\\&\quad
-\alpha_0(n\delta_m^{-n}+2(2n+1)\delta_m^{-n-1})\cdot t_0\, .
\end{aligned}
\end{equation*}
Of course, the elements $t_\infty$ and $t_0$ are central and
we have anti-symmetry.
The local cocycle, i.e. the cocycle coming with $t_\infty$ was given
in
\cite{Schlmunich} and reproduced in \cite[Equ. 12.75]{Schlkn}.
Unfortunately, there the central terms related to
$[H_n^{(.)},H_m^{(.)}]$ were forgotten. 
\begin{remark}
By Cox and Jurisich \cite[Thm.2.4]{CoJu} a different
form of a universal central extension for the $\sln(2,\C)$ current
algebra was proposed. This form was taken up in \cite{CJM}.
An inspection of the structure equation shows that in the proposed
form
two independent cocycles coming with the central elements
$\omega_0$ and $\omega_1$ show up. Both would be local. But this
contradicts
the uniqueness of the local cocycle class (up to rescaling)
as obtained in \cite{Schlaff}, which was also  recalled in the
current
article.
A closer examination shows that if ``$\omega_1\ne 0$'' the 
proposed structure constants do not define a Lie algebra. 
The reader might check himself, that e.g. the 
Jacobi identity of the triple
$(f_{-(n+m+2)}^1,h_m^1,e_n^1)$ (in the notation of the quoted 
articles) will be a non-zero multiple of $\omega_1$.
\end{remark}

The principal structure, as far as the central extension is concerned,
in particular also that we have one unique local cocycle class
(up to rescaling) does not depend in an essential manner on
the simple Lie algebra. See also
Bremner \cite{Brem3} for the example of the  
4-point case. Here the universal central extension is 3-dimensional,
One of the classes will be local with respect to the
standard splitting, the other two not. The latter two are ``coupled''
with ultrasperical (Gegenbauer) polynomials.
I like also to mention that in \cite{BeTer} also the 3-point case was
considered
in another basis exhibiting another symmetry useful in the context of
statistical mechanics.

\begin{remark}\label{R:gln}
As an additional example we like to give the case of
the current algebra of $\gl(n,\C)$ for the $N$-point case.
Of course, as  $\gl(n,\C)$ is not perfect it does not admit
a universal central extension. But by the classification results
we can give the maximal central extension for which the cocycles
are multiplicative (or $\La$-invariant)
\begin{equation}
\begin{aligned}{}
[x\otimes f,y\otimes g]&=
[x,y]\otimes f\cdot g+
\sum_{i=1}^{N-1}\alpha_i\cdot \tr(x\cdot y)\res_{a_i}(fdg)\cdot t_i
\\ &\qquad
+
\sum_{i=1}^{N-1}\beta_i\cdot \tr(x)\tr(y)\res_{a_i}(fdg)\cdot s_i\,.
\end{aligned}
\end{equation}
Here $\alpha_i,\beta_i\in\C$ and $t_i$ and $s_i$ are central.
In this context see \cite[Sect.~9.8]{Schlkn}.
\end{remark}

\section{Projective and affine connections}
\label{sec:connect}

Let $\ (U_\al,z_\al)_{\al\in J}\ $ be a covering of the Riemann surface
$\Sigma_g$ 
by holomorphic coordinates with transition functions
$z_\be=f_{\be\al}(z_\al)$.

\begin{definition}
(a) A system of local (holomorphic, meromorphic) functions 
$\ R=(R_\al(z_\al))\ $ 
is called a (holomorphic, meromorphic) {\it projective 
connection} if it transforms as
\begin{equation}\label{eq:pc}
R_\be(z_\be)\cdot (f_{\beta,\alpha}')^2=R_\al(z_\al)+S(f_{\beta,\alpha}),
\qquad\text{with}\quad
S(h)=\frac {h'''}{h'}-\frac 32\left(\frac {h''}{h'}\right)^2,
\end{equation}
the Schwartzian derivative.
Here ${}'$ denotes differentiation with respect to
the coordinate $z_\al$.

(b) A system of local (holomorphic, meromorphic) functions 
$\ T=(T_\al(z_\al))\ $ 
is called a (holomorphic, meromorphic) {\it affine 
connection} if it transforms as
\begin{equation}\label{eq:zc}
T_\be(z_\be)\cdot (f_{\beta,\alpha}')=T_\al(z_\al)+
\frac{f''_{\beta,\alpha}}{f'_{\beta,\alpha}}.
\end{equation}
\end{definition}

Every Riemann surface admits a holomorphic projective 
connection  \cite{HawSchiff},\cite{Gun}.
Given a point $P$ then there exists always a meromorphic
 affine connection holomorphic outside of $P$ and having maximally a
 pole
of order one there \cite{SchlDiss}.

 From their very definition 
  it follows that the difference of two affine (projective) connections
 will be a (quadratic) differential.
 Hence, after fixing one affine (projective) connection all others are obtained
 by adding (quadratic) differentials.

\bibliographystyle{amsplain}

\end{document}